\documentclass[english]{article}
\usepackage{graphicx} 

\usepackage[T1]{fontenc}
\usepackage[utf8]{inputenc}
\usepackage{lmodern}
\usepackage{amssymb,authblk}
\usepackage{mathtools}
\usepackage{amsthm}
\usepackage{stmaryrd}
\usepackage{mathrsfs}
\usepackage{amsfonts}
\usepackage{faktor}
\usepackage{xcolor}
\usepackage{tcolorbox}
\usepackage[linewidth=1pt,
middlelinecolor= black,
middlelinewidth=0.4pt,
roundcorner=1pt,
topline = false,
rightline = false,
bottomline = false,
rightmargin=0pt,
skipabove=0pt,
skipbelow=0pt,
leftmargin=-1cm,
innerleftmargin=1cm,
innerrightmargin=0pt,
innertopmargin=0pt,
innerbottommargin=0pt]{mdframed}

\usepackage{multicol}
\usepackage{array}
\usepackage[a4paper]{geometry}
\usepackage{shuffle}
\usepackage[english]{babel}
\usepackage{hyperref}
\usepackage{tikz-cd}
\usepackage{fp}
\usepackage{ifthen}
\usepackage{calc}
\usetikzlibrary{automata, positioning, arrows,cd}

\theoremstyle{definition}
\newtheorem{defi}{Definition}[section]
\newtheorem{Eg}[defi]{\textbf{Example}}

\newtheorem*{Eg*}{Example}
\newtheorem*{defi*}{Definition}
\newtheorem*{Rq*}{\textbf{Remarque}}

\newtheorem{innercustomgeneric}{\customgenericname}
\providecommand{\customgenericname}{}
\newcommand{\newcustomtheorem}[2]{%
	\newenvironment{#1}[1]
	{%
		\renewcommand\customgenericname{#2}%
		\renewcommand\theinnercustomgeneric{##1}%
		\innercustomgeneric
	}
	{\endinnercustomgeneric}
}
\newcustomtheorem{customdefi}{Dfinition}
\newcustomtheorem{customrq}{Remark}
\newcustomtheorem{customEgs}{Examples}
\newcustomtheorem{customEg}{Example}

\theoremstyle{plain}
\newtheorem{Prop}[defi]{Proposition}
\newtheorem{Lemme}[defi]{Lemma}

\newtheorem{thm}[defi]{Theorem}

\newtheorem*{thm*}{Théorème}
\newtheorem*{Lemme*}{Lemme}
\newtheorem*{Prop*}{Proposition}
\newtheorem{remark}[defi]{Remark}

\newtheorem{innercustomgenerictwo}{\customgenericname}
\providecommand{\customgenericname}{}
\newcommand{\newcustomtheoremplain}[2]{%
	\newenvironment{#1}[1]
	{%
		\renewcommand\customgenericname{#2}%
		\renewcommand\theinnercustomgenerictwo{##1}%
		\innercustomgenerictwo
	}
	{\endinnercustomgeneric}
}
\newcustomtheoremplain{customthm}{Théorème}
\newcustomtheoremplain{customlemma}{Lemme}
\newcustomtheoremplain{customprop}{Proposition}

\newcommand{\K}{\mathbb{K}}
\newcommand{\N}{\mathbb{N}}

\newcommand{\R}{\mathbb{R}}

\newcommand{\tdun}[1]{\begin{picture}(10,5)(-2,-1)
		\put(0,0){\circle*{2}}
		\put(2,-2){\tiny #1}
\end{picture}}

\newcommand{\tun}{\begin{picture}(10,5)(-2,-1)
		\put(0,0){\circle*{2}}
\end{picture}}

\newcommand{\tddeux}[2]{\begin{picture}(15,26)(-5,-1)
		\put(3,0){\circle*{2}} 
		\put(3,0){\line(0,1){7}}
		\put(3,7){\circle*{2}} 
		
		\put(5,-1){\tiny #1}
		\put(5,6){\tiny #2}
\end{picture}}

\newcommand{\tdtroisun}[3]{\begin{picture}(20,12)(-5,-1)
		\put(3,0){\circle*{2}}
		\put(-0.65,0){$\vee$}
		\put(6,7){\circle*{2}}
		\put(0,7){\circle*{2}}
		\put(5,-2){\tiny #1}
		\put(7,5){\tiny #2}
		\put(-5,8){\tiny #3}
\end{picture}}

\newcommand{\tdquatredeux}[4]{\begin{picture}(20,20)(-5,-1)
		\put(3,0){\circle*{2}}
		\put(-.65,0){$\vee$}
		\put(6,7){\circle*{2}}
		\put(0,7){\circle*{2}}
		\put(0,14){\circle*{2}}
		\put(0,7){\line(0,1){7}}
		\put(5,-2){\tiny #1}
		\put(9,5){\tiny #2}
		\put(-5,5){\tiny #3}
		\put(-5,12){\tiny #4}
\end{picture}}

\newcommand{\tdquatrequatre}[4]{\begin{picture}(20,14)(-5,-1)
		\put(3,5){\circle*{2}}
		\put(-.65,5){$\vee$}
		\put(6,12){\circle*{2}}
		\put(0,12){\circle*{2}}
		\put(3,0){\circle*{2}}
		\put(3,0){\line(0,1){5}}
		\put(6,-3){\tiny #1}
		\put(6,4){\tiny #2}
		\put(9,12){\tiny #3}
		\put(-5,12){\tiny #4}
\end{picture}}

\newcommand\restr[2]{{
  \left.\kern-\nulldelimiterspace 
  #1 
  \littletaller 
  \right|_{#2} 
  }}

  \newcommand{\littletaller}{\mathchoice{\vphantom{\big|}}{}{}{}}

\newcommand {\calF}{{\mathcal{F}}}

\newcommand {\calT}{{\mathcal {T}}}

\newcommand {\calW}{{\mathcal {W}}}

\begin{document}

\title{Coalgebras, bialgebras and Rota-Baxter algebras from shuffles of rooted forests}
\author{Pierre~J.~Clavier${}^{1}$, Douglas Modesto${}^{1}$\\
~\\
\normalsize \it $^1$  Department of Mathematics, IRIMAS, \\
\normalsize \it Université de Haute Alsace.\\
~\\
\normalsize email: pierre.clavier@uha.fr}

\date{}

\maketitle

\begin{abstract} 
 We construct and study new generalisations to rooted trees and forests of some properties of shuffles of words. First, we build a coproduct on rooted trees which, together with their shuffle, endow them with bialgebra structure. We then caracterize the coproduct dual to the shuffle product of rooted forests and build a product on rooted trees to obtain the bialgebra dual to the shuffle bialgebra. We then characterize and enumerate primitive trees for the dual coproduct. Finally, using modified shuffles of rooted forests, we prove a property in the category of Rota-Baxter algebras.
\end{abstract}

\tableofcontents

\section*{Introduction}

\addcontentsline{toc}{section}{Introduction}

\subsection*{Forests, trees and words}

\addcontentsline{toc}{subsection}{Forests, trees and words}

Rooted trees and forests are a stepstone of combinatorics, and their various properties have been studied by generations of mathematicians. Not only are they essential in pure combinatorics (see for example \cite{Foissy01}), but they are also appear in other domain of mathematics, and in particular number theory (see \cite{Ma13,Cl20}) which is the motivation for this work. Let us also mention in passing that they are also crucial in Physics. They turned out for example to be at the heart of renormalisation theory (\cite{CK1,CK2}), to play an essential role in the study of Schwinger-Dyson equations (\cite{foissy2014general} and also \cite{balduf2024tubings} for recent advances in this domain). 

What will interest us here is that rooted trees and forests are a natural generalisation of words. As it turns out, words as well have many important properties that play important roles in combinatorics and number theory. For a set $\Omega$, we write $\calW_\Omega$ the vector space freely generated by words written in the alphabet $\Omega$. The starting point of our analysis is the shuffle of words, and its contracting counterparts, which are sometimes called quasi-shuffles, contracting shuffles or stuffles. The following definition is taken from \cite{Ho00} but is equivalent to other definition of shuffles, for example in \cite{Guo}.
\begin{defi} \label{def:shuffles}
 Let $\K$ be a field, $A$ an algebra over $\K$ and $\lambda\in\K$. The \textbf{$\lambda$-shuffle product} on $\calW_A$ is defined recursively as follows:
 \begin{enumerate}
			\item $\emptyset \shuffle_\lambda \omega = \omega \shuffle_\lambda \emptyset = \omega$ for any word $w\in\calW_A$;
			\item for two non-empty words $\omega = w_0 \sqcup \tilde{\omega}$ and $\omega' = w'_0 \sqcup \tilde{\omega'}$ in $\calW_A$ we set 
			\begin{equation*}
			\omega \shuffle_\lambda \omega' = w_0 \sqcup(\tilde{\omega}\shuffle_\lambda \omega') + w_0' \sqcup (\omega \shuffle_\lambda \tilde{\omega'}) + \lambda(w_0 w'_0)\sqcup(\tilde{\omega} \shuffle_\lambda \tilde{\omega'})
			\end{equation*}
  where $w_0 w'_0\in A$ stands for the multiplication of $w_0$ and $w_0'$ in $A$ and $\sqcup$ for the concatenation of words in $\calW_A$.
 \end{enumerate}
\end{defi}
It is a standard exercise to show that the $\shuffle_\lambda$ are associative products for all $\lambda\in\K$. If $A$ is commutative, then the $\lambda$-shuffle are also commutative. Definition \ref{def:shuffles} still makes sense if $A$ is only a semigroup. For $A$ a set, we need to take $\lambda=0$. In the following, we will not differentiate these cases.

These products have many important properties which we now list. The first one is that one can actually build a coproduct which endows $\calW_A$ with a bialgebra structure. We refer the reader to one of the many excellent introductions to the theory of bialgebras and Hopf algebras (for example \cite{manchon2004hopf}) for full definition. For our purpose, let us simply say if $(B,m)$ is a $\K-$algebra with product $m:B\otimes B\longrightarrow B$ associative, we are looking for a {\bf coproduct} map $\Delta:B\longrightarrow B\otimes B$ which should be coassociative, and a {\bf counit} map $\varepsilon:B\longrightarrow\K$. These maps should have the following properties:
\begin{equation}\label{eq:bialg}
\begin{tikzcd}
\mathbb{K}\otimes B & B\otimes B \arrow[r, "I\otimes \varepsilon"] \arrow[l, "\varepsilon\otimes Id"'] & B\otimes \mathbb{K} \\
                    & B \arrow[u, "\Delta"] \arrow[ru, "\sim"] \arrow[lu, "\sim"]                      &                    
\end{tikzcd},
\begin{tikzcd}
B\otimes B \arrow[d, "\Delta \otimes \Delta"] \arrow[rr, "m"]          &  & B \arrow[d, "\Delta"] \\
B\otimes B\otimes B\otimes B \arrow[rr, "(m\otimes m)\circ \tau_{23}"] &  & B\otimes B          
\end{tikzcd}.
    \end{equation}
If the case of words, we take $B=\calW_A$ and $\Delta$ is the deconcatenation coproduct. The dual of the $0$-shuffle product is another coproduct on (the dual of) $\calW_A$. This coproduct, called deshuffle, together with the concatenation product endow (the dual of) $\calW_A$ with another bialgebra structure.

Finally, we also need to mention that words, together with the $\lambda$-shuffle products, have a universal property in the categoryof Rota-Baxter algebras \cite{baxter1960analytic,rota1969baxter}. This universal property was proven in \cite{Guo,GUO2000117} and will be presented later in this article.

\subsection*{Content and main results}

\addcontentsline{toc}{subsection}{Content and main results}

The $\lambda$-shuffle products of Definition \ref{def:shuffles} not only endow words with rich algebraic structures but also play important roles in the theory of multi zeta values (MZVs). These objects were generalised in \cite{Ma13} and \cite{Cl20} to rooted trees and forests. In \cite{Cl20}, generalisations of $\lambda$-shuffles to rooted forests were built that play the same roles to generalised MZVs as usual $\lambda$-shuffles to MZVs. Since rooted trees and forests generalise words, our goal in this paper is to generalise the aformentioned properties of $\lambda$-shuffles of words to the shuffles products of rooted trees and forests.

We start Section \ref{sec:one} by introducing necessary combinatorial definitions and in particular the definition of $\lambda$-shuffles of rooted forests (Definition \ref{def:shuffle_forests}). The associated coproduct is defined in Definition \ref{def:shuffle_coprod}. It is counital and coassociative (Propositin \ref{prop:shuffle_coalg}). Our main result of this first section is Theorem \ref{thm:shuffle_bialg} which states that the shuffles of rooted trees together with this coproduct endow the space of rooted trees with a bialgebra structure.

Section \ref{sec:two} is devoted to the study of structures dual to the shuffle of rooted forests. We give an inductive caracterisation of the coproduct dual to the shuffle of rooted forests which is stated in Theorem \ref{thm:coprod_dual_ind}, our first main result of this second section. We then introduce some further combinatorial constructions on rooted trees in Definition \ref{def:adm_family}, and in particular admissible families of vertices. The second main result of this second section is then Theorem \ref{thm:coprod_adm} which provide a purely combinatorial description of the dual coproduct in terms of admissible families of vertices. 

Moving on, we introduce two grafting products (Definition \ref{def:GL_like_products}) which are pre-Lie (Proposition \ref{prop:pre_Lie}). The third important result of this section is that one of this pre-Lie product, together with the dual coproduct, endows the dual space of rooted trees with a bialgebra structure (Theorem \ref{thm:dual_bialg}), which is the dual to the bialgebra structure of the aforementioned Theorem \ref{thm:shuffle_bialg}. We finish this second section with a description of trees that are primitive for the dual coproduct (Theorem \ref{thm:description_prim_trees}). This description implies an inductive formula for the number of these primitive trees (Proposition \ref{prop:prim_trees}). This formula allows to easily compute the number of primitive trees with a low number of vertices: we give the values for trees up to 23 vertices, obtained with a standard laptop.

In section \ref{sec:three}, the final section of this paper, we look at the relationship between shuffles of trees and Rota-Baxter algebras. We first recall some notions of the theory of Rota-Baxter algebras, and a particular a universal property of the shuffle products of words. We introduce in Definition \ref{def:new_shuffle} new shuffles of rooted trees and forests. These are shown in our last important result, namely Theorem \ref{thm:quasi_univ_prop}, to have a property of universal type in the category of Rota-Baxter algebras.

\section{Shuffle bialgebra} \label{sec:one}

\subsection{Shuffle of rooted forests} \label{subsec:shuffle_coalg}

Let us start by recalling some classical definitions of (oriented) graphs theory. These definitions can be found in many introductions to the topic, for example \cite{Wi96}.
\begin{defi}
 \begin{itemize}
  \item A {\bf graph} is a pair of finite sets $G:=(V(G),E(G))$ with $E(G)\subseteq V(G)\times V(G)$. $E(G)$ is the set of edges of the graph and $V(G)$ the set of vertices of the graph. 
  \item A {\bf path} in a graph $G$ is a finite sequence of elements of $V(G)$: $p=(v_1,\cdots,v_n)$ such that for all $i\in[[1,n-1]]$, $(v_i,v_{i+1})$ is an edge of $G$.  
  By convention, there is always a path between a vertex and itself. 
  \item For a graph $G=(V(G);E(G))$, let $\preceq$ be the binary relation on $V(G)$ defined by: $v_1\preceq v_2$ if, and only if, it exists a path from $v_1$ to $v_2$. We also denote by $\succeq$ the inverse relation. A {\bf directed acyclic graph} (DAG for short) is a graph such that $(V(G),\preceq)$ is a poset.
  
  \item A {\bf forest} is a DAG such that there is at most one path between two vertices. A {\bf rooted forest} is a forest whose connected components each have a unique minimal element. These elements are called {\bf roots}. A {\bf rooted tree} is a connected rooted forest.
  \item Let $\Omega$ be a set. An {\bf $\Omega$-decorated rooted forest} is a rooted forest $F$ together with a {\bf decoration map} 
  $d:V(F)\mapsto\Omega$. 
  \item Two rooted forests $F$ and $F'$ (resp. decorated rooted forests $(F,d_F)$ and $(F',d_{F'})$) are {\bf isomorphic} if there exists a poset isomorphism $f_V:V(F)\longrightarrow V(F')$ (resp. and $d_F = d_{F'}\circ f_V$).
 \end{itemize}
\end{defi}
As usual, we always consider isomorphism classes of rooted forests and therefore identify trees and forests with their classes and write $\calF_\Omega$ the set of isomorphism classes of $\Omega$-decorated rooted forests and $\calF$ the set of isomorphism classes of undecorated rooted forests.  Furthermore, when there is no need to specify the decoration map we simply write $F$ for a decorated forest $(F,d)$. We also allow ourself to identify a vertex and its decoration in order to not have too heavy notations.

We write $\emptyset$ the empty forest. Notice that it belongs to $\calF_\Omega$ and set $\calF^+_\Omega:=\calF_\Omega\setminus\{\emptyset\}$. Finally, for a (decorated or not) rooted forest $F$, we write $|F|$ the number of vertices of $F$ (thus $|\emptyset|=0$). This degree induces a graduation of $\calF$ and $\calF_\Omega$.

Our rooted forests are non-planar. This implies in particular that their concatenation is a commutative (and associative) product. We write $F_1\cdots F_k$ the concatenation of the forests $F_1,\cdots,F_k$. The empty forest is the unit of this product.

Further recall that for any $\omega\in\Omega$ the {\bf grafting operator} $B_+^\omega:\calF_{\Omega}\longrightarrow\calT_\Omega$ is the linear operator that, to any rooted forest $F=T_1\cdots T_k$, associates the decorated tree obtained from $F$ by adding a root decorated by $\omega$ linked to each root of $T_i$ for $i$ going from $1$ to $k$.\\

The next definition introduces the central object of our inquiry. It is taken from \cite{Cl20} and is a generalisation to rooted forests of the shuffle products of words (Definition \ref{def:shuffles}).
\begin{defi} \label{def:shuffle_forests}
 Let $(\Omega,.)$ be a commutative semigroup and $\lambda\in\K$. The {\bf $\lambda$-shuffle} of two rooted forests $F$ and $F'$ is defined recursively on $N=|F|+|F'|$ by
 \begin{itemize}
  \item  If $|F|+|F'|=0$ (and thus $F=F'=\emptyset$), we set $\emptyset\shuffle^T_\lambda \emptyset = \emptyset$.
 
  \item For $N\in\N$, assume the shuffle product of forests has been defined on every forests $f,f'$ such that $|f|+|f'|\leq N$. Then for any two forests $F,F'$ such that $|F|+|F'|=N+1$;
  \begin{enumerate}
   \item If $F'=\emptyset$ (or $F=\emptyset$) set $F\shuffle^T_\lambda \emptyset=\emptyset\shuffle^T_\lambda F=F$.
   \item  If $F$ or $F'$ is not a tree, then we can write $F$ and $f$ uniquely as a concatenation of trees: $F=T_1\cdots T_k$ and $F'=t_1\cdots t_n$ with the $T_i$s and $t_j$s 
  nonempty, $k+n\geq3$ and set 
  \begin{equation} \label{eq:def_shuffle_forests}
   F\shuffle^T_\lambda F' = \frac{1}{kn}\sum_{i=1}^k\sum_{j=1}^n\left((T_i\shuffle^T_\lambda t_j)T_1\cdots\widehat{T_i}\cdots T_n t_1\cdots\widehat{t_j}\cdots t_k\right)
  \end{equation}
  where $T_1\cdots\widehat{T_i}\cdots T_n$ stands for the concatenation of the trees $T_1,\cdots,T_n$ without the tree $T_i$.
   \item If $F=T = B_+^a(f)$ and $F'=T'=B_+^{a'}(f')$ are two nonempty trees, we set 
  \begin{equation} \label{eq:def_shuffle_trees}
   T\shuffle^T_\lambda T' = B_+^a(f\shuffle^T_\lambda  T') + B_+^{a'}(T\shuffle^T_\lambda  f') +\lambda B_+^{a.a'}(f\shuffle^T_\lambda f').
  \end{equation}
  \end{enumerate}
 \end{itemize}
 This product is then extended by bilinearity to a product on $\calF_\Omega$.
\end{defi}
\begin{remark} \label{rk:shuffle_trees}
 If $\lambda=0$, $\Omega$ can be a set without a semigroup structure. In this case, we write $\shuffle^T$ for $\shuffle^T_0$ and call this product the {\bf shuffle product} of rooted forests. This is the case we will focus on in the next Section \ref{sec:two}.
\end{remark}
Before moving on to an coalgebra associated to the shuffle product it might be relevant to write some examples of shuffle products of rooted forests.
\begin{Eg}
 $\big)\big)$
\end{Eg}

Finally, the following result of \cite{Cl20} will play an important role in the rest of this paper.
\begin{Prop}
        Let $\lambda \in \R$ and $(\Omega,\cdot )$ be a commutative semi-group. Then $(\mathcal{F},\shuffle^T_\lambda, \emptyset)$ is an nonassociative, commutative, unital algebra. 
    \end{Prop}

\subsection{Shuffle coproduct}


Since the canonical injection $j_\Omega:\calW_\Omega\hookrightarrow\calF_\Omega$ send words to linear trees which are connected forests, we define a generalisation of the deconcatenation coproduct on rooted trees.
\begin{defi} \label{def:shuffle_coprod}
Define $\Delta:\calT_\Omega\longrightarrow\calT_\Omega\otimes\calT_\Omega$ inductively on the number of vertices. First set $\Delta(\emptyset)=\emptyset\otimes\emptyset$. Assume that $\Delta$ has been defined on all trees with $N$ vertices and let $T$ be a tree with $N+1$ vertices. Then it exists $a\in\Omega$ such that $T=B_+^a(F)$. We have two cases to consider:
\begin{itemize}
 \item If $F=t\in\calT_\Omega$ set
 \begin{equation} \label{eq:def_coprod_shuffle}
  \Delta(B_+^a(t)) = (B_+^a \otimes Id)\Delta(t) + \emptyset\otimes B_+^a(t).
 \end{equation}
 \item If $F=f \in \mathcal{F}_\Omega \setminus \mathcal{T}_\Omega$ set 
 \begin{equation*}
  \Delta(B_+^a(f)) = \emptyset \otimes B_+^a(f).
 \end{equation*}
\end{itemize}
This map is then extended by linearity to a map $\Delta:\calT_\Omega\longrightarrow\calT_\Omega\otimes\calT_\Omega$.
\end{defi}
This coproduct admits a simple graphical representation. Let the {\bf trunk} of a non-empty rooted tree be the path from the root to the first vertex with at least two descendant. Then $\Delta(T)$ is obtained by cutting the tree $T$ below each vertex in the trunk. The lower parts are the first elements and the top parts are the second elements of $\Delta(T)=\sum_{(T)}T'\otimes T''$, where we use Sweedler's notation for the coproduct. In the case of a linear trees, the initial hypothesis $\Delta(\emptyset)=\emptyset\otimes\emptyset$ brings one more cut above the unique leaf of the tree. In particular $\Delta$ coincides with the deconcatenation coproduct on linear trees.

We first have that $\Delta$ endow $\calT_\Omega$ with a coalgebra structure.
\begin{Prop} \label{prop:shuffle_coalg}
     $(\mathcal{T}_\Omega,\Delta)$ is a noncocommutative, coassociative, counital coalgebra.
 \end{Prop}
 \begin{proof}

  The counity map $\varepsilon:\calT_\Omega\longrightarrow\K$ is given by the linear map which acts on rooted trees by $\varepsilon(\emptyset)=1$ and $\varepsilon(T)=0$ for any non-empty rooted tree $T$. This trivially fulfills the axioms of the counit (first diagramm of \eqref{eq:bialg}).
  
  It is obvious from the definition that if $T$ is not empty and not a linear tree, then $\Delta(T)$ is not symmetric, thus $\Delta$ is non cocommutative. We are only left to prove that $\Delta$ is coassociative. \\
 
  We proceed by induction on $|T|$ for $T\in \mathcal{T}_\Omega$. If $|T| = 0,$ we have $T = \emptyset$ and 
     $$(\Delta \otimes Id)\Delta(\emptyset) = \emptyset\otimes \emptyset \otimes \emptyset = (Id \otimes \Delta)\Delta(\emptyset).$$
  Suppose now that $\Delta$ is coassociative for all $t\in \mathcal{T}_\Omega$ with $|t|\leq N$ and let $T \in \mathcal{T}_\Omega$ such that $|T| = N+1.$ We have two cases, first if $T \in \mathcal{T}_\Omega$ be such that $T = B_+^a(f)$ with $f\in \mathcal{F}_\Omega\setminus \mathcal{T}_\Omega,$ then
     $$\Delta(T) = \emptyset \otimes T$$
     and
     $$(\Delta \otimes Id)\Delta(T) = \emptyset\otimes \emptyset \otimes T,$$
     $$(Id \otimes \Delta)\Delta(T) = \emptyset \otimes \emptyset\otimes T.$$
     Second, if $T = B_+^a(t)$ with $t\in\calT_\Omega$, let $\Delta(t) = \sum_{(t)}t_1 \otimes t_2.$ Then we have
     $$\Delta(T) = \sum_{(t)}B_+^a(t_1)\otimes t_2 + \emptyset\otimes T.$$
     It follows that
     $$(\Delta \otimes Id)\Delta(T) = \sum_{(t)} \left((B_+^a\otimes Id)\Delta(t_1) + \emptyset \otimes B_+^a(t_1)\right)\otimes t_2 + \emptyset\otimes \emptyset \otimes T.$$
  Let us focus on the first term of this expression. We have
  \begin{equation*}
   \sum_{(t)} \left((B_+^a\otimes Id)\Delta(t_1)\right)\otimes t_2=(B_+^a\otimes Id\otimes Id)\sum_{(t)}\Delta(t_1)\otimes t_2 = (B_+^a\otimes Id\otimes Id)(\Delta\otimes Id)\Delta(t).
  \end{equation*}
  Since $|t|=N$ we can use the induction hypothesis: we have $(\Delta\otimes Id)\Delta(t)=(id \otimes\Delta)\Delta(t)$. Thus we obtain 
  \begin{equation*}
   \sum_{(t)} \left((B_+^a\otimes Id)\Delta(t_1)\right)\otimes t_2=(B_+^a\otimes Id\otimes Id)\sum_{(t)}t_1\otimes \Delta(t_2) = \sum_{(t)}B_+^a(t_1)\otimes\Delta(t_2).
  \end{equation*}
  Thus we obtain 
  \begin{equation*}
   (\Delta \otimes Id)\Delta(T) = \sum_{(t)}\left(B_+^a(t_1)\otimes\Delta(t_2)+\emptyset \otimes B_+^a(t_1)\otimes t_2\right) + \emptyset\otimes \emptyset \otimes T.
  \end{equation*}
  And on the other hand 
  \begin{align*}
   (Id\otimes \Delta)\Delta(T) & =\sum_{(t)}(Id\otimes \Delta)\left(B_+^a(t_1)\otimes t_2\right) +\emptyset\otimes \Delta(T) \\
   & = \sum_{(t)}B_+^a(t_1)\otimes\Delta(t_2) +\emptyset\otimes\left(\sum_{(t)}B_+^a(t_1)\otimes t_2 + \emptyset\otimes T\right) \\
   & = (\Delta \otimes Id)\Delta(T)
  \end{align*}
  as expected. Therefore $\Delta$ is coassociative.
 \end{proof}
 
 \subsection{Shuffle bialgebra}
 
 Let us now state and prove the main result of this section. Notice that, as pointed out before, if $\lambda=0$, we can have $\Omega$ to be a set without a semigroup structure in this Theorem (see Remark \ref{rk:shuffle_trees}).
 \begin{thm} \label{thm:shuffle_bialg}
  Let $(\Omega,.)$ be a commutative semigroup. For any $\lambda\in\K$, $(\mathcal{T}_\Omega,\shuffle_\lambda^T,\Delta)$ is a commutative, non-associative, non-cocommutative and coassociative  bialgebra. 
 \end{thm}
 \begin{proof}
  The various properties of the products and coproducts have been proven in \cite{Cl20} and in Proposition \ref{prop:shuffle_coalg} above. Then compatibility of the counit with the product and of the unit with the coproduct are trivial to check. Thus we only need to prove that the coproduct is an morphism for the shuffle product (second diagram of \eqref{eq:bialg}). We want to prove that, for any $\lambda$ in $\K$ and any rooted trees $T$ and $T'$ one has
  \begin{equation}\label{morphshuf}
         \Delta(T\shuffle^T_\lambda T') = (\shuffle^T_\lambda \otimes \shuffle^T_\lambda)\circ \tau_{23}\circ (\Delta \otimes \Delta)(T\otimes T').
  \end{equation}
  We prove this by induction on $N=|T|+|T'|$. If $N=0$ we have $T=T'=\emptyset$ and a trivial computation gives
  \begin{equation*}
   \Delta (\emptyset \shuffle^T_\lambda \emptyset) =\emptyset\otimes\emptyset= (\shuffle^T_\lambda \otimes \shuffle^T_\lambda)\circ \tau_{23}\circ (\Delta \otimes \Delta)(\emptyset \otimes \emptyset).
  \end{equation*}
  Assume now that Equation \eqref{morphshuf} holds for any $t,t' \in \mathcal{T}_\Omega$ with $|t|+|t'| \leq N,$ and let $T,T' \in \mathcal{T}$ such that $|T|+|T'| = N+1.$ We have various cases to consider:
  
   If $T$ (or $T'$) is empty, then $\Delta(T \shuffle^T_\lambda \emptyset) = \Delta(T)$ and
     $$ (\shuffle^T_\lambda \otimes \shuffle^T_\lambda)\circ \tau_{23}\circ (\Delta \otimes \Delta) (T \otimes \emptyset) = (\shuffle^T_\lambda \otimes \shuffle^T_\lambda)\circ \tau_{23}\left(\sum_{(T)}T_1\otimes T_2 \otimes \emptyset \otimes 
     \emptyset\right)=\sum_{(T)}T_1\otimes T_2 =\Delta(T) $$
     and Equation \eqref{morphshuf} holds at rank $N+1$ in this case as expected.
   
   If $T$ and $T'$ are both nonempty we can write $T=B_+^a(F)$ and $T'=B_+^a(F')$ and we have four subcases to consider depending on whether or not $F$ and $F'$ are rooted trees.
   \begin{enumerate}
    \item If $T = B_+^a(t)$ and $T' = B_+^{a'}(t')$ with $t$ and $t'$ in $\calT_\Omega$, let $\Delta(t) = \sum_{(t)}t_1\otimes t_2$ and $\Delta(t') = \sum_{(t')}t'_1\otimes t'_2.$
    We have then
    $$T\shuffle^T_\lambda T' = B_+^a(t\shuffle^T_\lambda T') + B_+^{a'}(T\shuffle^T_\lambda t') + \lambda B_+^{a.a'}(t\shuffle^T_\lambda t').$$
    By definition of the $\lambda$-huffle products of rooted forests, shuffles of rooted trees are linear combinations of rooted trees. Therefore
    \begin{align*}
     & \Delta(t\shuffle^T_\lambda T') = (B_+^a\otimes Id)\Delta(t\shuffle^T_\lambda T ') + \emptyset\otimes B_+^a(t\shuffle^T_\lambda T') \\
     = ~& (B_+^a\otimes Id)(\shuffle^T_\lambda \otimes \shuffle^T_\lambda)\circ \tau_{23}\left(\Delta(t)\otimes\Delta(T')\right) + \emptyset\otimes B_+^a(t\shuffle^T_\lambda T') \quad\text{by the induction hypothesis}\\
     = ~& (B_+^a\otimes Id)(\shuffle^T_\lambda \otimes \shuffle^T_\lambda)\circ \tau_{23}\left(\Delta(t)\otimes\left((B_+^{a'}\otimes Id)\Delta(t') +\emptyset\otimes T'\right)\right) + \emptyset\otimes B_+^a(t\shuffle^T_\lambda T') \\
     = ~& (B_+^a\otimes Id)(\shuffle^T_\lambda \otimes \shuffle^T_\lambda)\circ \tau_{23}\left(\sum_{(t),(t')}t_1\otimes t_2\otimes B_+^{a'}(t'_1)\otimes t'_2 + \sum_{(t)}t_1\otimes t_2\otimes\emptyset\otimes T'\right) + \emptyset\otimes B_+^a(t\shuffle^T_\lambda T') \\
     = ~& \sum_{(t),(t')}B_+^a(t_1\shuffle^T_\lambda B_+^{a'}(t'_1))\otimes (t_2\shuffle^T_\lambda t'_2) + \sum_{(t)}B_+^a(t_1)\otimes (t_2\shuffle^T_\lambda T') + \emptyset\otimes B_+^a(t\shuffle^T_\lambda T').
    \end{align*}
    Performing the same computation for the other two terms of $\Delta(T\shuffle_\lambda^T T')$ we find
    \begin{align*}
       & \Delta(T\shuffle^T_\lambda T') = \sum_{(t),(t')}B_+^a(t_1\shuffle^T_\lambda B_+^{a'}(t'_1))\otimes (t_2\shuffle^T_\lambda t'_2) + \sum_{(t)}B_+^a(t_1)\otimes (t_2\shuffle^T_\lambda T') + \emptyset\otimes B_+^a(t\shuffle^T_\lambda T')\\
    + & \sum_{(t),(t')}B_+^{a'}(B_+^a(t_1)\shuffle^T_\lambda t'_1)\otimes (t_2\shuffle^T_\lambda t'_2) + \sum_{(t')}B_+^{a'}(t'_1)\otimes (T\shuffle^T_\lambda t'_2) + \emptyset\otimes B_+^{a'}(T\shuffle^T_\lambda t')\\
             + & \lambda\sum_{(t),(t')}B_+^{a.a'}(t_1\shuffle^T_\lambda t'_1)\otimes (t_2\shuffle^T_\lambda t'_2) + \lambda\emptyset \otimes B_+^{a.a'}(t\shuffle^T_\lambda t').
    \end{align*}
    Notice than in each of these computations, we could use the induction hypothesis since each of the trees appearing in $t\shuffle^T_\lambda T'$, $T\shuffle^T_\lambda t'$ and  $t\shuffle^T_\lambda t'$ have at most $N$ vertices. Putting the three summations over $(t)$ and $(t')$ together we obtain
    \begin{align*}
     & \Delta(T\shuffle^T_\lambda T') = \sum_{(t),(t')}\left(B_+^a(t_1\shuffle^T_\lambda B_+^{a'}(t'_1)) + B_+^{a'}(B_+^a(t_1)\shuffle^T_\lambda t'_1) + \lambda B_+^{a.a'}(t_1\shuffle^T_\lambda t'_1)\right) \otimes (t_2\shuffle^T_\lambda t'_2) \\
    + & \sum_{(t)}B_+^a(t_1)\otimes (t_2\shuffle^T_\lambda T') + \sum_{(t')}B_+^{a'}(t'_1)\otimes (T\shuffle^T_\lambda t'_2)+ \emptyset\otimes\left(B_+^a(t\shuffle^T_\lambda T')+B_+^{a'}(T\shuffle^T_\lambda t') + \lambda B_+^{a.a'}(t\shuffle^T_\lambda t')\right) \\
     = & \sum_{(t),(t')}\left(B_+^a(t_1)\shuffle^T_\lambda B_+^{a'}(t'_1)\right) \otimes (t_2\shuffle^T_\lambda t'_2) + \sum_{(t)}B_+^a(t_1)\otimes (t_2\shuffle^T_\lambda T') + \sum_{(t')}B_+^{a'}(t'_1)\otimes (T\shuffle^T_\lambda t'_2) \\
     & \hspace{10cm} \llcorner+ \emptyset \otimes (T\shuffle^T_\lambda T').
    \end{align*}
   On the other hand, we also have
   \begin{align*}
    & \Delta(T)\otimes\Delta(T')  = \left((B_+^a\otimes Id)\Delta(t)+\emptyset\otimes T\right)\otimes\left((B_+^{a'}\otimes Id)\Delta(t')+\emptyset\otimes T'\right) \\
     = & \sum_{(t),(t')}B_+^a(t_1)\otimes t_2\otimes B_+^{a'}(t'_1)\otimes t_2 + \sum_{(t)}B_+^a(t_1)\otimes t_2\otimes\emptyset\otimes T' + \sum_{(t')}\emptyset \otimes T\otimes B_+^{a'}(t'_1)\otimes t'_2 + \emptyset\otimes T \otimes\emptyset\otimes T'.
   \end{align*}
   We then have
   \begin{align*}
               (\shuffle^T_\lambda \otimes \shuffle^T_\lambda & )\circ \tau_{23} (\Delta \otimes \Delta)(T\otimes T') = \sum_{(t),(t')}\left(B_+^a(t_1)\shuffle^T_\lambda B_+^{a'}(t'_1)\right)\otimes (t_2\shuffle^T_\lambda t'_2)  \\
              & \llcorner + \sum_{(t)}B_+^a(t_1)\otimes (t_2\shuffle^T_\lambda T')+ \sum_{(t')}B_+^{a'}(t'_1)\otimes (T\shuffle^T_\lambda t'_2) + \emptyset\otimes (T\shuffle^T_\lambda T') \\
              =~ & \Delta(T\shuffle^T_\lambda T')
          \end{align*}
    therefore Equation \eqref{morphshuf} holds in this case.
    \item Next, consider the case $T = B_+^a(f)$ with $f$ a non-connected forest and $T' = B_+^{a'}(t')$. As before, let us write $\Delta(t') = \sum_{(t')}t'_1\otimes t'_2.$ In this case we have
          $$T\shuffle^T_\lambda T' = B_+^a(f\shuffle^T_\lambda T') + B_+^{a'}(T \shuffle^T_\lambda t') + \lambda B_+^{a.a'}(f\shuffle^T_\lambda t').$$
    Since $f$ is a non-connected forest, $f\shuffle^T_\lambda T'$ and $f\shuffle^T_\lambda t'$ are non-connected as well. Therefore taking their coproduct only amounts to tensoring these terms on the left by $\emptyset$. For the second term, we perform the same type of computation than in the first case:
    \begin{align*}
     & \Delta B_+^{a'}(T \shuffle^T_\lambda t') = (B_+^{a'}\otimes Id)\Delta(T\shuffle^T_\lambda t') + \emptyset\otimes B_+^{a'}(T\shuffle^T_\lambda t') \\
     = ~& (B_+^{a'}\otimes Id)(\shuffle^T_\lambda \otimes \shuffle^T_\lambda )\circ \tau_{23} \left(\Delta(T)\shuffle^T_\lambda \Delta(t')\right) + \emptyset\otimes B_+^{a'}(T\shuffle^T_\lambda t')  \quad\text{by the induction hypothesis}\\
     = ~ & (B_+^{a'}\otimes Id)(\shuffle^T_\lambda \otimes \shuffle^T_\lambda )\circ \tau_{23} \left(\sum_{(t')}\emptyset\otimes T \otimes t'_1\otimes t'_2\right) + \emptyset\otimes B_+^{a'}(T\shuffle^T_\lambda t') \\
     = ~ & \sum_{(t')} B_+^{a'}(t'_1)\otimes(T\shuffle^T_\lambda t'_2) + \emptyset\otimes B_+^{a'}(T\shuffle^T_\lambda t').
    \end{align*}
    We then find
    \begin{align*}
     \Delta(T\shuffle^T_\lambda T') & = \sum_{(t')}B_+^{a'}(t'_1)\otimes (T\shuffle^T_\lambda t') + \emptyset\otimes \left(B_+^{a}(f\shuffle^T_\lambda T')+ B_+^{a'}(T\shuffle^T_\lambda t') + B_+^{a}(f\shuffle^T_\lambda t')\right) \\
     & = \sum_{(t')}B_+^{a'}(t'_1)\otimes (T\shuffle^T_\lambda t') + \emptyset\otimes (T\shuffle^T_\lambda T').
    \end{align*}
    On the other hand:
    \begin{equation*}
     \Delta(T)\otimes\Delta(T')= \emptyset\otimes T\otimes\left((B_+^{a'}\otimes Id)\Delta(t')+\emptyset\otimes T'\right) = \sum_{(t')}\emptyset\otimes T\otimes B_+^{a'}(t'_1)\otimes t'_2+\emptyset\otimes T\otimes\emptyset\otimes T'.
    \end{equation*}
    We then have
$$(\shuffle^T_\lambda \otimes \shuffle^T_\lambda)\circ \tau_{23}\circ (\Delta \otimes \Delta)(T\otimes T') = \sum_{(t')}B_+^{a'}(t'_1)\otimes (T\shuffle^T_\lambda t'_2) + \emptyset\otimes (T\shuffle^T_\lambda T')=\Delta(T\shuffle^T_\lambda T')$$
and Equation \eqref{morphshuf} holds in this case as well.
    \item The case $T=B_+^a(t)$ and $T'=B_+^{a'}(f')$ reduces to the previous case by symmetry of $\shuffle^T_\lambda$.
    \item Finally, if $T = B_+^a(f)$ and $T' = B_+^{a'}(f')$ with $f$ and $f'$ two non-connected forests we have 
            $$T\shuffle^T_\lambda T' = B_+^a(f\shuffle^T_\lambda T') + B_+^{a'}(T \shuffle^T_\lambda f') + \lambda B_+^{a.a'}(f\shuffle^T_\lambda f').$$
    Each of the terms in the grafting operators in the RHS is a non-connected forest therefore we have 
    \begin{equation*}
     \Delta(T\shuffle^T_\lambda T') = \emptyset\otimes B_+^a(f\shuffle^T_\lambda T') + \emptyset\otimes B_+^{a'}(T \shuffle^T_\lambda f') + \lambda \emptyset\otimes B_+^{a.a'}(f\shuffle^T_\lambda f') = \emptyset\otimes (T\shuffle^T_\lambda T').
    \end{equation*}
    And on the other hand
    \begin{equation*}
     (\shuffle^T_\lambda \otimes \shuffle^T_\lambda)\circ \tau_{23}\circ (\Delta \otimes \Delta)(T\otimes T') =(\shuffle^T_\lambda \otimes \shuffle^T_\lambda)\circ \tau_{23}(\emptyset\otimes T\emptyset\otimes T') = \emptyset\otimes (T\shuffle^T_\lambda T').
    \end{equation*}
    Therefore Equation \eqref{morphshuf} holds in this case as well.
   \end{enumerate}
   Thus Equation \eqref{morphshuf} holds for all tree $T$ and $T'$ such that $|T|+|T'|=N+1$ and therefore, by induction, we have that Equation \eqref{morphshuf} holds for any trees $T$ and $T'$. Therefore, $(\mathcal{T}_\Omega,\shuffle_\lambda^T,\Delta)$ is a bialgebra as claimed.
 \end{proof}
\begin{remark}
     Due to the fact that $\shuffle^T_\lambda$ is not associative, the standard construction to build an antipode from a connected graded bialgebra does not apply here. And indeed, it is not possible to find a antipode that turns this bialgebra into a Hopf algebra. 

To see this, let us notice that we must have  $S(\emptyset) = \emptyset$. Thus for $S$ to be an antipode we need to have
\begin{align*}
    & \shuffle \circ (Id \otimes S)\circ \Delta (\tdtroisun{a}{b}{c}) = \shuffle \circ (S \otimes Id)\circ \Delta (\tdtroisun{a}{b}{c}) = \mu\circ \varepsilon(\tdtroisun{a}{b}{c}) \\
    \Longleftrightarrow~ & \underbrace{\emptyset \shuffle S(\tdtroisun{a}{b}{c})}_{S(\tdtroisun{a}{b}{c})} = \underbrace{S(\emptyset)\shuffle \tdtroisun{a}{b}{c}}_{\tdtroisun{a}{b}{c}} = 0.
\end{align*}
Therefore there is no map $S: \mathcal{T}_\Omega \to \mathcal{T}_\Omega$ that turns this bialgebra into a Hopf algebra.

     However one can define $S : \mathcal{T}_\Omega \to \mathcal{T}_\Omega$ such that
     $$S(T) = \begin{cases}
         0, \quad \text{if} \quad |\text{leaves}(T)|>1\\
         (-1)^{n+1}T^{-1}, \quad \text{otherwise}
     \end{cases}$$
     where $n = |T|$ and $T^{-1}$ is the tree with the same vertices than $T$ but with their order reversed.
     This map makes the following diagram commute
   \begin{equation*}
\begin{tikzcd}
\mathcal{T}_\Omega \arrow[r, "\epsilon"] \arrow[d, "\Delta"]           & \mathbb{K} \arrow[r, "\mu"] & \mathcal{T}_\Omega                                                 \\
\mathcal{T}_\Omega\otimes \mathcal{T}_\Omega \arrow[rr, "Id\otimes S"] &                             & \mathcal{T}_\Omega\otimes \mathcal{T}_\Omega \arrow[u, "\shuffle"]
\end{tikzcd}
   \end{equation*} 
     We call such a map a right antipode.
\end{remark}
 Computation suggest there are might be other coproducts that turn $(\calT_\Omega,\shuffle^T_\lambda)$ into a bialgebra. A classification is beyond the scope of this work and left for future studies. Furthermore, we will see in the next section that this bialgebra is canonical in some sense.
 
 \section{Dual bialgebra} \label{sec:two}
 
 In this section, we specialise our computations to the case $\lambda=0$: we will focus the coproduct $\Delta^*$ dual to $\shuffle^T$.

\subsection{The dual coproduct}

Let $(A,m)$ be a algebra freely generated as a vector space by a finite set $X$. Then
recall from \cite{sweedler1969hopf} that for any such finite dimensional algebra the product $m$ induces a coproduct $\Delta^*$ on the dual $A^*$. $\Delta^*$ is called the {\bf dual coproduct}; that is to say ``dual to $m$''. It is defined as follows: for a basis element $x\in X$, let us write $\langle x,.\rangle$ the dual element in $A^*$ defined on elements of the basis by  $\langle x,y\rangle=\delta_{xy}$ where $\delta_{xy}$ is the Kronecker symbol. Extending the maps $\langle x,.\rangle$ by linearity gives us a basis of $A^*$, and we define
\begin{align*}
     \Delta^* :  A^* &\to (A\otimes A)^* \simeq A^*\otimes A^*\\
      f &\mapsto \Delta^*(f)
 \end{align*}
by
\begin{equation*}
 \langle\Delta^*(f),a\otimes b\rangle= \langle f,m(a\otimes b)\rangle,\quad a,b \in A\otimes A.
\end{equation*}
It is a standard exercise (see for example \cite{manchon2004hopf}) that $\Delta^*$ is cocommutative (resp. coassociative) if, and only if, $m$ is commutative (resp. associative).

In the case of rooted forest, we can use the fact that $\calF$ is graded by the number of vertices: $\calF=\bigoplus_{n=0}^{\infty}\calF_n$ with $\calF_n$ finite dimensional for all $n\in\N$. We work with the {\bf graded dual} $\calF^\circ:=\bigoplus_{n=0}^{\infty}\calF_n^*$. This also work if the forests are decorated by a finite set $\Omega$, or if $\Omega$ is such that each $(\calF_{\Omega})_n$ is itself a graded, with each degree being finite dimensional. From now on, we will assume this to be true.

As in the generic case, a basis of $\calF^\circ$ is given by the the maps $\delta_F$ on rooted forests by $\delta_F(F)=1$ and $\delta_F(F')=0$ if $F\neq F'$, and extended by linearity to a map on $\calF$. Then $\calF$ and $\calF^\circ$ are clearly isomorphic as vector spaces and therefore, as a short hand notation, we will from now on write $F\in\calF^\circ$ instead of $\delta_F\in\calF^\circ$. We use the same notations in the decorated case.\\

From now on, $\Delta^*$ will be the coproduct dual to the shuffle of rooted forests. Recall also that for a rooted forest $F = T_1\dots T_n$ and $I\subseteq [n]$ we write $F_I:=\sqcup_{i\in I}T_i$. We prove a recursive formula for this coproduct, which for rooted trees is reminiscent of the Hochschild cocycle property of Connes-Kreimer coproduct (see for example \cite{connes1999hopf}) which is important for applications of this coproduct to other domains, e.g. SPDEs (see for example \cite{hairer2015geometric}).
\begin{thm} \label{thm:coprod_dual_ind}
 We have $\Delta^*(\emptyset) = \emptyset\otimes\emptyset$, and for any $n\geq2$, set $[n]_i:=[n]\setminus\{i\}$. Then for any non-connected forest $F = T_1\dots T_n$
 \begin{equation} \label{eq:Delta_dual_forest}
    \Delta^*(F) = F\otimes \emptyset + \emptyset\otimes F + \sum_{i = 1}^{n}(\sqcup\otimes\sqcup)\circ\tau_{23}\left(\tilde{\Delta}^*(T_i)\otimes \left(\sum_{I\subset [n]_i} \alpha_{I,n}F_I\otimes F_{[n]_i\setminus I}\right)\right)
\end{equation}
with $\sqcup$ the concatenation of rooted forests, $\tau_{23}$ the usual flip defined by $\tau_{23}(a\otimes b\otimes c\otimes d)=a\otimes c\otimes b\otimes d$ and 
\begin{equation*}
 \alpha_{I,n} := (|I|+1)(n - |I|).
\end{equation*}
For a tree $T = B_+^a(F)$ we have
\begin{equation}\label{coproddualtree}
    \Delta^*(T) = (B_+^a\otimes Id)(Id\otimes \Pi_{\mathcal{T}_\Omega})\Delta^*(F) + (Id\otimes B_+^a)(\Pi_{\mathcal{T}_\Omega}\otimes Id)\Delta^*(F).
\end{equation}
with $\Pi_{\mathcal{T}_\Omega} : \mathcal{F}_\Omega \to \mathcal{T}_\Omega$ the projection to rooted trees parallely to non-connected rooted forests.
\end{thm}
\begin{proof}
 First notice that $F_1\shuffle^T F_2=\emptyset\Longleftrightarrow F_1=F_2=\emptyset$. This proves $\Delta^*(\emptyset)=\emptyset\otimes\emptyset$.\\
 
 Let us now prove Equation \eqref{eq:Delta_dual_forest}. For $F\in\calF_\Omega^+$ we are looking for all forests $f_1$ and $f_2$ such that $\langle F,f_1\shuffle^T f_2\rangle\neq0$. First observe that we have the trivial solutions $f_1=F$, $f_2=\emptyset$ and $f_1=\emptyset$, $f_2=F$. These give us the terms $F\otimes \emptyset$ and $\emptyset\otimes F$ in Equation \eqref{eq:Delta_dual_forest}.
 
 For the non-trivial solutions, let us set $f^1=t_1\cdots t_p$ and $f^2=t'_1\cdots t'_q$. From Equation \eqref{eq:def_shuffle_forests} of the definition of the shuffle of rooted forests we must have $p+q=n+1$. We must also have that one of the tree constituting $F$ appears in $t_j\shuffle t'_k$, for some $j\in\{1,\cdots,p\}$ and $k\in\{1,\cdots,q\}$. The other trees of $F$ must then be the other trees of $f^1$ and $f^2$. Writing this rigorously be obtain
  \begin{equation*}
   \exists i\in\{1,\cdots,n\},j\in\{1,\cdots,p\},k\in\{1,\cdots,q\}\text{ s.t. }\langle T_i,t_j\shuffle t'_k\rangle\neq0~\wedge~ \langle F_{[n]_i},f^1_{[p]_j}f^2_{[q]_k}\rangle\neq0.
  \end{equation*}
  Since by assumption $t_j$ and $t'_k$ are not empty, the condition $\langle T_i,t_j\shuffle t'_k\rangle\neq0$ gives precisely that $\tilde\Delta^*(T_i)$ contains the tensor product $t_j\otimes t'_k$.
  
  Let us turn our attention to the condition $\langle F_{[n]_i},f^1_{[p]_j}f^2_{[q]_k}\rangle\neq0$. It is equivalent to saying to each tree in $F$ other than the i-th $T_i$ must appear either in $f^1$ or in $f^2$. Therefore we must split $F_{[n]_i}$ into two parts, eventually empty. Each of these part must then become either $f^1$ or $f^2$. Therefore we obtain the terms $F_I\otimes F_{[n]_i\setminus I}$ in Equation \eqref{eq:Delta_dual_forest}. We have constructed our rooted forests $f^1$ and $f^2$: they must be of the form
  \begin{equation} \label{eq:intermediate_computation}
   f^1\otimes f^2= (\sqcup\otimes\sqcup)\circ\tau_{23}\left(\tilde{\Delta}^*(T_i)\otimes \left(F_I\otimes F_{[n]_i\setminus I}\right)\right)
  \end{equation}
  for some $i\in[n]$ and $I\subseteq [n]_i$. However, observe that the coefficient $\frac{1}{nk}$ in Equation \eqref{eq:def_shuffle_forests} has not be taken into account. The left tensor in Equation \eqref{eq:intermediate_computation} has $|I|+1$ connected components, and the right tensor has $n-|I|$ connected components if $\tilde\Delta(T_i)\neq0$. Thus we find that, for the $f^1$ and $f^2$ in Equation \eqref{eq:intermediate_computation} we have 
  \begin{equation*}
   f^1\shuffle^T f^2=\frac{1}{(|I|+1)(n-|I|)}F+\cdots.
  \end{equation*}
  Thus, with the coefficients $\alpha_{I,n}$, we have the right formula for $\Delta^*(F)$. \\
  
   To prove Equation \eqref{coproddualtree} for $T=B^a_+(F)$ let us look for trees $T_1$ and $T_2$ such that $\langle T, T_1\shuffle^T T_2\rangle\neq0$. Indeed, by definition of the shuffle of rooted forests, we cannot have $\langle T, F_1\shuffle^T F_2\rangle\neq0$ if $F_1$ and $F_2$ are not a trees. First, observe that $(T_1,T_2)=(T,\emptyset)$ and $(T_1,T_2)=(\emptyset,T)$ are two (trivial) solutions.
   
   To find the non-trivial solutions, assume $T_1$ and $T_2$ are non-empty and write $T_i=B_+^{a_i}(F_i)$. Then Equation \eqref{eq:def_shuffle_trees} of the definition of the shuffle of rooted forests gives
   \begin{equation*}
    \langle T, T_1\shuffle^T T_2\rangle = \delta_{aa_1}\langle F, F_1\shuffle^T T_2\rangle + \delta_{aa_2}\langle F, T_1\shuffle^T F_2\rangle
   \end{equation*}
   with $\delta_{ax}$ the Kronecker symbol. Then we see that the tensor products appearing in $\Delta^*(T)$ with $T_1\neq\emptyset\neq T_2$ are precisely the tensor products appearing in
   \begin{equation*}
    (Id\otimes\Pi_{\calT_\Omega})(\Delta^*(F)-F\otimes\emptyset)+(\Pi_{\mathcal{T}_\Omega}\otimes Id)(\Delta^*(F)-\emptyset\otimes F).
   \end{equation*}
   We conclude by noticing that $T\otimes\emptyset=(B^a_+\otimes I)(Id\otimes\Pi_{\calT_\Omega})(F\otimes\emptyset)$ and $\emptyset\otimes T=(Id\otimes B_+^a)(\Pi_{\mathcal{T}_\Omega}\otimes Id)(\emptyset\otimes F)$. Thus putting everything together we obtain
   \begin{align*}
    \Delta^*(T) & = T\otimes\emptyset+\emptyset\otimes T+(Id\otimes\Pi_{\calT_\Omega})(\Delta^*(F)-F\otimes\emptyset)+(\Pi_{\mathcal{T}_\Omega}\otimes Id)(\Delta^*(F)-\emptyset\otimes F) \\
    & = (Id\otimes\Pi_{\calT_\Omega})\Delta^*(F)+(\Pi_{\mathcal{T}_\Omega}\otimes Id)\Delta^*(F)
   \end{align*}
   as claimed.
\end{proof}
The inductive formula for $\Delta^*$ given by Theorem \ref{thm:coprod_dual_ind} allows to easily compute coproducts of rooted forests of reasonnable sizes, as the next example shows.
\begin{Eg}
Using the inductive formulas of Theorem \ref{thm:coprod_dual_ind} we obtain the following expressions for coproducts. \\
$\Delta^*\Big($
 $\otimes$ \tdun{e}.

\end{Eg}

\subsection{Admissible families}

One of the remarquable feature of Connes-Kreimer coproduct is that it admits a purely combinatorial description, in terms of \emph{admissible cuts} (see for example \cite{Foissy01} for the definition). We now give a similar combinatorial description of the 
dual coproduct $\Delta^*$. For this we will need to introduce some new combinatorial concepts, but first let us recall that for a rooted tree $T$, the {\bf fertility} of a vertex $v\in V(T)$ of $T$ is the number of edges $(v,v')\in E(T)$, i.e. the number of direct descendant of the vertex $v$.
\begin{defi} \label{def:adm_family}
 \begin{enumerate}
  \item Let $T\in\calT_\Omega$ be a rooted tree and $\Lambda \subseteq V(T)$ a subset of vertices of $T.$ We define {\bf the tree induced by $\Lambda$} to be the rooted tree $T_\Lambda \in \mathcal{T}_\Omega$ as the tree $(V(T_\Lambda), E(T_\Lambda))$ such that $V(T_\Lambda) = \Lambda$ and
$$E(T_\Lambda) = \{(v_1,v_2)  \in V(T_\Lambda)^2| v_1 <_T v_2 \; \text{and} \; \nexists v_3\in\Lambda  \; \text{such that} \; v_1 <_T v_3 <_T v_2 \}.$$
\item For $T$ a rooted tree, $\Gamma\subseteq V(T)$ is called {\bf admissible} or an {\bf admissible family} when we simultaneously have:
    \begin{itemize}
    \item $T_\Gamma$ and $T_{V(T)\setminus\Gamma}$ are both rooted trees;
    \item for every node $v \in \Gamma$ with $\text{fert}(v)\geq2,$  we have  $\text{fert}(v)$ in $T_{\Gamma}$ to be equal to $\text{fert}(v)$ in $T.$
    \item for every node $v \in V(T)\setminus\Gamma$ with $\text{fert}(v)\geq2,$ we have  $\text{fert}(v)$ in $T_{V(T)\setminus\Gamma}$ to be equal to $\text{fert}(v)$ in $T.$
    \end{itemize}
    \item Let $T$ be a rooted tree and $\Gamma\subseteq V(T)$. We define {\bf contraction of $T$ with respect to $\Gamma$} to be the undecorated forest  $ \text{cont}_{\Gamma}(T)\in\calT$ obtained by contracting all the subtrees $S\subseteq T$ such that
    \begin{equation*}
     V(S)\subset \Gamma\text{ and }\nexists v \in {V(T)\setminus\Gamma}\text{ such that }root(S) <_T v
    \end{equation*}
    and contracting also all the subtrees $S\subseteq T$ such that
    \begin{equation*}
     V(S)\subseteq V(T)\setminus{\Gamma}\text{ and }\nexists v \in {\Gamma}\text{ such that }root(S) <_T v.
    \end{equation*}
    We also define the {\bf contracted fertility} (relative to $\Gamma$) to be $1$ if cont$_\Gamma(T)$ is empty or reduced to only one vertex and
     $$c_{\Gamma} = \prod_{\substack{v\in \text{cont}_{\Gamma}(T) \\ v\notin \ell(\text{cont}_{\Gamma}(T) )}} \text{fert}(v)$$
     otherwise.
\end{enumerate}
\end{defi}
\begin{remark} \label{rk:fertility_relative}
 Let us point out that $c_\Gamma$ depends on the trees in which we take $\Gamma$, in the sense that if we have $\Gamma\subseteq V(T')\subseteq V(T)$ where $T'$ is a subtree of $T$, then $c_\Gamma$ is ambiguous since it is not the same if we see $\Gamma$ as a subset of $V(T')$ or of $V(T)$. We do not write $c_{\Gamma,T}$ since it is rather cumbersome to write and it should be clear from context in which trees we are taking $\Gamma$ to be.
\end{remark}
Notice that for any tree $T$, $\emptyset$ and $V(T)$ are admissible families of $T$. In both case we have cont$_\Gamma(T)=\tun$ if $T$ is not empty.

Also notice that, by construction $T_\Lambda$ is only a rooted forest, but we use this notation since we will focus on admissible subset of vertices. Furthermore, in general we do not have $E(T_\Lambda)\subseteq E(T)$. This is illustrated in the following example:
\begin{Eg}
Let $T$ be the following tree:
    \begin{equation*}
T = 
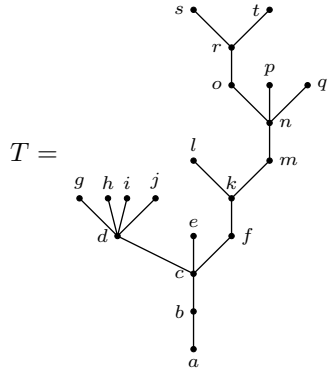
\begin{tikzpicture}[line cap=round,line join=round,>=triangle 45,x=0.5cm,y=0.5cm,baseline=(base)]
            \node (base) at (0,6) {};
            \filldraw[color=black] (0,1)circle(1pt);
            \filldraw[color=black] (0,2)circle(1pt);
            \filldraw[color=black] (0,3)circle(1pt);
            \filldraw[color=black] (-2,4)circle(1pt);
            \filldraw[color=black] (0,4)circle(1pt);
            \filldraw[color=black] (1,4)circle(1pt);
            \filldraw[color=black] (-3,5)circle(1pt);
            \filldraw[color=black] (-2.25,5)circle(1pt);
            \filldraw[color=black] (-1.75,5)circle(1pt);
            \filldraw[color=black] (-1,5)circle(1pt);
            \filldraw[color=black] (1,5)circle(1pt);
            \filldraw[color=black] (0,6)circle(1pt);
            \filldraw[color=black] (2,6)circle(1pt);
            \filldraw[color=black] (2,7)circle(1pt);
            \filldraw[color=black] (1,8)circle(1pt);
            \filldraw[color=black] (2,8)circle(1pt);
            \filldraw[color=black] (3,8)circle(1pt);
            \filldraw[color=black] (1,9)circle(1pt);
            \filldraw[color=black] (0,10)circle(1pt);
            \filldraw[color=black] (2,10)circle(1pt);
			\draw [line width=.5pt] (0.,1.)-- (0,2.);
			\draw [line width=.5pt] (0.,2.)-- (0,3.);
            \draw [line width=.5pt] (0.,3.)-- (-2,4.);
            \draw [line width=.5pt] (0.,3.)-- (0,4.);
            \draw [line width=.5pt] (0.,3.)-- (1,4.);
            \draw [line width=.5pt] (-2.,4.)-- (-3,5.);
            \draw [line width=.5pt] (-2.,4.)-- (-2.25,5.);
            \draw [line width=.5pt] (-2.,4.)-- (-1.75,5.);
            \draw [line width=.5pt] (-2.,4.)-- (-1,5.);
            \draw [line width=.5pt] (1.,4.)-- (1,5.);
            \draw [line width=.5pt] (1.,5.)-- (0,6.);
            \draw [line width=.5pt] (1.,5.)-- (2,6.);
            \draw [line width=.5pt] (2.,6.)-- (2,7.);
            \draw [line width=.5pt] (2.,7.)-- (1,8.);
            \draw [line width=.5pt] (2.,7.)-- (2,8.);
            \draw [line width=.5pt] (2.,7.)-- (3,8.);
            \draw [line width=.5pt] (1.,8.)-- (1,9.);
            \draw [line width=.5pt] (1.,9.)-- (0,10.);
            \draw [line width=.5pt] (1.,9.)-- (2,10.);
			\draw (0,1) node[below]{\scriptsize{$a$}};
            \draw (0,2) node[left]{\scriptsize{$b$}};
            \draw (0,3) node[left]{\scriptsize{$c$}};
            \draw (-2,4) node[left]{\scriptsize{$d$}};
            \draw (0,4) node[above]{\scriptsize{$e$}};
            \draw (1,4) node[right]{\scriptsize{$f$}};
            \draw (-3,5) node[above]{\scriptsize{$g$}};
            \draw (-2.25,5) node[above]{\scriptsize{$h$}};
            \draw (-1.75,5) node[above]{\scriptsize{$i$}};
            \draw (-1,5) node[above]{\scriptsize{$j$}};
            \draw (1,5) node[above]{\scriptsize{$k$}};
            \draw (0,6) node[above]{\scriptsize{$l$}};
            \draw (2,6) node[right]{\scriptsize{$m$}};
            \draw (2,7) node[right]{\scriptsize{$n$}};
            \draw (1,8) node[left]{\scriptsize{$o$}};
            \draw (2,8) node[above]{\scriptsize{$p$}};
            \draw (3,8) node[right]{\scriptsize{$q$}};
            \draw (1,9) node[left]{\scriptsize{$r$}};
            \draw (0,10) node[left]{\scriptsize{$s$}};
            \draw (2,10) node[left]{\scriptsize{$t$}};
\end{tikzpicture}
    \end{equation*}
The set $\Gamma = \{a,c,d,e,g,h,i,j,k,l,m,r,s,t\}$ gives us a admissible family of $T.$
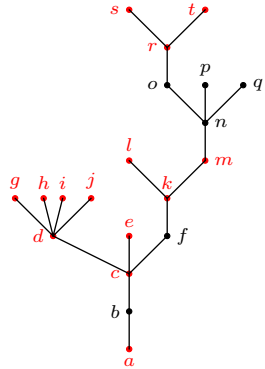
\begin{equation*}
\begin{tikzpicture}[line cap=round,line join=round,>=triangle 45,x=0.5cm,y=0.5cm,baseline=(base)]
            \node (base) at (0,6) {};    
            
            \filldraw[color=red] (0,1)circle(1pt);
            \filldraw[color=black] (0,2)circle(1pt);
            \filldraw[color=red] (0,3)circle(1pt);
            \filldraw[color=red] (-2,4)circle(1pt);
            \filldraw[color=red] (0,4)circle(1pt);
            \filldraw[color=black] (1,4)circle(1pt);
            \filldraw[color=red] (-3,5)circle(1pt);
            \filldraw[color=red] (-2.25,5)circle(1pt);
            \filldraw[color=red] (-1.75,5)circle(1pt);
            \filldraw[color=red] (-1,5)circle(1pt);
            \filldraw[color=red] (1,5)circle(1pt);
            \filldraw[color=red] (0,6)circle(1pt);
            \filldraw[color=red] (2,6)circle(1pt);
            \filldraw[color=black] (2,7)circle(1pt);
            \filldraw[color=black] (1,8)circle(1pt);
            \filldraw[color=black] (2,8)circle(1pt);
            \filldraw[color=black] (3,8)circle(1pt);
            \filldraw[color=red] (1,9)circle(1pt);
            \filldraw[color=red] (0,10)circle(1pt);
            \filldraw[color=red] (2,10)circle(1pt);
			\draw [line width=.5pt] (0.,1.)-- (0,2.);
			\draw [line width=.5pt] (0.,2.)-- (0,3.);
            \draw [line width=.5pt] (0.,3.)-- (-2,4.);
            \draw [line width=.5pt] (0.,3.)-- (0,4.);
            \draw [line width=.5pt] (0.,3.)-- (1,4.);
            \draw [line width=.5pt] (-2.,4.)-- (-3,5.);
            \draw [line width=.5pt] (-2.,4.)-- (-2.25,5.);
            \draw [line width=.5pt] (-2.,4.)-- (-1.75,5.);
            \draw [line width=.5pt] (-2.,4.)-- (-1,5.);
            \draw [line width=.5pt] (1.,4.)-- (1,5.);
            \draw [line width=.5pt] (1.,5.)-- (0,6.);
            \draw [line width=.5pt] (1.,5.)-- (2,6.);
            \draw [line width=.5pt] (2.,6.)-- (2,7.);
            \draw [line width=.5pt] (2.,7.)-- (1,8.);
            \draw [line width=.5pt] (2.,7.)-- (2,8.);
            \draw [line width=.5pt] (2.,7.)-- (3,8.);
            \draw [line width=.5pt] (1.,8.)-- (1,9.);
            \draw [line width=.5pt] (1.,9.)-- (0,10.);
            \draw [line width=.5pt] (1.,9.)-- (2,10.);
			\draw (0,1) node[below][color=red]{\scriptsize{$a$}};
            \draw (0,2) node[left]{\scriptsize{$b$}};
            \draw (0,3) node[left][color=red]{\scriptsize{$c$}};
            \draw (-2,4) node[left][color=red]{\scriptsize{$d$}};
            \draw (0,4) node[above][color=red]{\scriptsize{$e$}};
            \draw (1,4) node[right]{\scriptsize{$f$}};
            \draw (-3,5) node[above][color=red]{\scriptsize{$g$}};
            \draw (-2.25,5) node[above][color=red]{\scriptsize{$h$}};
            \draw (-1.75,5) node[above][color=red]{\scriptsize{$i$}};
            \draw (-1,5) node[above][color=red]{\scriptsize{$j$}};
            \draw (1,5) node[above][color=red]{\scriptsize{$k$}};
            \draw (0,6) node[above][color=red]{\scriptsize{$l$}};
            \draw (2,6) node[right][color=red]{\scriptsize{$m$}};
            \draw (2,7) node[right]{\scriptsize{$n$}};
            \draw (1,8) node[left]{\scriptsize{$o$}};
            \draw (2,8) node[above]{\scriptsize{$p$}};
            \draw (3,8) node[right]{\scriptsize{$q$}};
            \draw (1,9) node[left][color=red]{\scriptsize{$r$}};
            \draw (0,10) node[left][color=red]{\scriptsize{$s$}};
            \draw (2,10) node[left][color=red]{\scriptsize{$t$}};
\end{tikzpicture}
\end{equation*}
From $\Gamma$ we can build the following trees
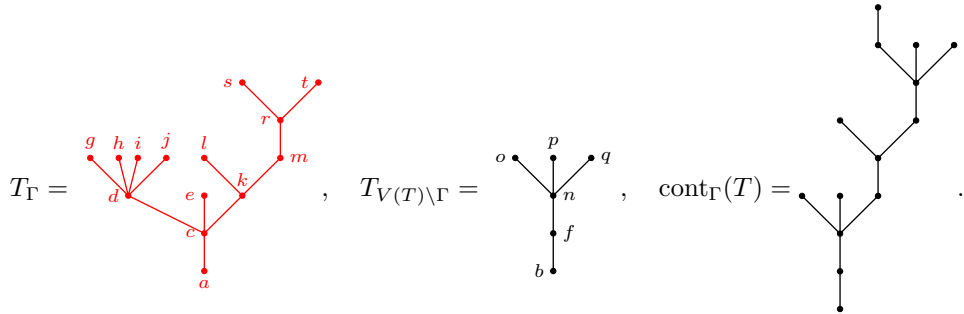
\begin{equation*}
 T_\Gamma = \begin{tikzpicture}[line cap=round,line join=round,>=triangle 45,x=0.5cm,y=0.5cm, baseline=(base)]
            \node (base) at (0,3) {};
            \filldraw[color=red] (0,1)circle(1pt);
            
            \filldraw[color=red] (0,2)circle(1pt);
            \filldraw[color=red] (-2,3)circle(1pt);
            \filldraw[color=red] (0,3)circle(1pt);
            
            \filldraw[color=red] (-3,4)circle(1pt);
            \filldraw[color=red] (-2.25,4)circle(1pt);
            \filldraw[color=red] (-1.75,4)circle(1pt);
            \filldraw[color=red] (-1,4)circle(1pt);
            \filldraw[color=red] (1,3)circle(1pt);
            \filldraw[color=red] (0,4)circle(1pt);
            \filldraw[color=red] (2,4)circle(1pt);
            
            \filldraw[color=red] (2,5)circle(1pt);
            \filldraw[color=red] (1,6)circle(1pt);
            \filldraw[color=red] (3,6)circle(1pt);
			\draw [line width=.5pt][color=red] (0.,1.)-- (0,2.);
			\draw [line width=.5pt][color=red] (0.,2.)-- (0,3.);
            \draw [line width=.5pt][color=red] (0.,2.)-- (-2,3.);
            \draw [line width=.5pt][color=red] (0.,2.)-- (1,3.);
            \draw [line width=.5pt][color=red] (-2.,3.)-- (-3,4);
            \draw [line width=.5pt][color=red] (-2.,3.)-- (-2.25,4);
            \draw [line width=.5pt][color=red] (-2.,3.)-- (-1.75,4);
            \draw [line width=.5pt][color=red] (-2.,3.)-- (-1,4);
            \draw [line width=.5pt][color=red] (1.,3.)-- (0,4.);
            \draw [line width=.5pt][color=red] (1.,3.)-- (2,4.);
            \draw [line width=.5pt][color=red] (2.,4.)-- (2,5.);
            \draw [line width=.5pt][color=red] (2.,5.)-- (3,6.);
            \draw [line width=.5pt][color=red] (2.,5.)-- (1,6.);
            
			\draw (0,1) node[below][color=red]{\scriptsize{$a$}};
           
            \draw (0,2) node[left][color=red]{\scriptsize{$c$}};
            \draw (-2,3) node[left][color=red]{\scriptsize{$d$}};
            \draw (0,3) node[left][color=red]{\scriptsize{$e$}};
            
            \draw (-3,4) node[above][color=red]{\scriptsize{$g$}};
            \draw (-2.25,4) node[above][color=red]{\scriptsize{$h$}};
            \draw (-1.75,4) node[above][color=red]{\scriptsize{$i$}};
            \draw (-1,4) node[above][color=red]{\scriptsize{$j$}};
            \draw (1,3) node[above][color=red]{\scriptsize{$k$}};
            \draw (0,4) node[above][color=red]{\scriptsize{$l$}};
            \draw (2,4) node[right][color=red]{\scriptsize{$m$}};
            
            \draw (2,5) node[left][color=red]{\scriptsize{$r$}};
            \draw (1,6) node[left][color=red]{\scriptsize{$s$}};
            \draw (3,6) node[left][color=red]{\scriptsize{$t$}};
\end{tikzpicture}, \quad T_{V(T)\setminus \Gamma} =  \begin{tikzpicture}[line cap=round,line join=round,>=triangle 45,x=0.5cm,y=0.5cm,baseline=(base)]
            \node (base) at (0,3) {};
            
            \filldraw[color=black] (0,1)circle(1pt);

            \filldraw[color=black] (0,2)circle(1pt);
            
            \filldraw[color=black] (0,3)circle(1pt);
            \filldraw[color=black] (-1,4)circle(1pt);
            \filldraw[color=black] (0,4)circle(1pt);
            \filldraw[color=black] (1,4)circle(1pt);

			\draw [line width=.5pt] (0.,1.)-- (0,2.);
			\draw [line width=.5pt] (0.,2.)-- (0,3.);
            \draw [line width=.5pt] (0.,3.)-- (-1,4.);
            \draw [line width=.5pt] (0.,3.)-- (0,4.);
            \draw [line width=.5pt] (0.,3.)-- (1,4.);
            
			
            \draw (0,1) node[left]{\scriptsize{$b$}};
           
            \draw (0,2) node[right]{\scriptsize{$f$}};
           
            \draw (0,3) node[right]{\scriptsize{$n$}};
            \draw (-1,4) node[left]{\scriptsize{$o$}};
            \draw (0,4) node[above]{\scriptsize{$p$}};
            \draw (1,4) node[right]{\scriptsize{$q$}};
         
\end{tikzpicture},\quad
\text{cont}_\Gamma (T) = \begin{tikzpicture}[line cap=round,line join=round,>=triangle 45,x=0.5cm,y=0.5cm, baseline=(base)]
            \node (base) at (0,4) {};
            \filldraw[color=black] (0,1)circle(1pt);
            \filldraw[color=black] (0,2)circle(1pt);
            \filldraw[color=black] (0,3)circle(1pt);
            \filldraw[color=black] (-1,4)circle(1pt);
            \filldraw[color=black] (0,4)circle(1pt);
            \filldraw[color=black] (1,4)circle(1pt);
            \filldraw[color=black] (1,5)circle(1pt);
            \filldraw[color=black] (0,6)circle(1pt);
            \filldraw[color=black] (2,6)circle(1pt);
            \filldraw[color=black] (2,7)circle(1pt);
            \filldraw[color=black] (1,8)circle(1pt);
            \filldraw[color=black] (2,8)circle(1pt);
            \filldraw[color=black] (3,8)circle(1pt);
            \filldraw[color=black] (1,9)circle(1pt);
			\draw [line width=.5pt] (0.,1.)-- (0,2.);
			\draw [line width=.5pt] (0.,2.)-- (0,3.);
            \draw [line width=.5pt] (0.,3.)-- (-1,4.);
            \draw [line width=.5pt] (0.,3.)-- (0,4.);
            \draw [line width=.5pt] (0.,3.)-- (1,4.);
            \draw [line width=.5pt] (1.,4.)-- (1,5.);
            \draw [line width=.5pt] (1.,5.)-- (0,6.);
            \draw [line width=.5pt] (1.,5.)-- (2,6.);
            \draw [line width=.5pt] (2.,6.)-- (2,7.);
            \draw [line width=.5pt] (2.,7.)-- (1,8.);
            \draw [line width=.5pt] (2.,7.)-- (2,8.);
            \draw [line width=.5pt] (2.,7.)-- (3,8.);
            \draw [line width=.5pt] (1.,8.)-- (1,9.);

\end{tikzpicture}.
\end{equation*}
\end{Eg}
 We will give a combinatorial description of the coproduct $\Delta^*$ using the following inductive formulas for the admissible families and contracted fertility.
 \begin{Lemme} \label{lem:adm_fam_ind}
  \begin{enumerate}
   \item For a tree $T=B_+^a(t)$ with $t\in\calT_\Omega$, the admissible families of $T$ are the sets 
   \begin{equation*}
    \Gamma=\tilde\Gamma\quad\text{and}\quad\Gamma=\{a\}\cup\tilde\Gamma
   \end{equation*}
   with $\tilde\Gamma$ an admissible family of $t$. In both cases the reduced fertility is given by $c_\Gamma=c_{\tilde\Gamma}$.
   \item For a tree $T=B_+^a(T_1\cdots T_n)$ with $n\geq2$ the admissible families of $T$ are the sets
   \begin{equation*}
    \Gamma=\Gamma_i\quad\text{and}\quad\Gamma=\{a\}\cup\Gamma_i\cup V(T_1\cdots\widehat{T_i}\cdots T_n)
   \end{equation*}
   some some $i\in[n]$ and $\Gamma_i$ an admissible family of $T_i$ with $\Gamma_i\neq V(T_i)$. In both cases the reduced fertility is given by $c_\Gamma=n c_{\Gamma_i}$.
  \end{enumerate}
 \end{Lemme}
 \begin{proof}
  \begin{enumerate}
   \item By the definition of admissible families, we have that for a tree $T=B_+^a(t)$ with $t\in\calT_\Omega$, $\Gamma\subseteq V(T)$ is an admissible subfamily of $T$ if, and only if, $\Gamma\cap V(t)$ is an admissible family of $t$. Again by the definition of admissible families and since $t$ is a tree, we have that admissible families of $T$ might contain the root or not. In the former case, we find that admissible families of $T$ are exactly the sets $\{a\}\cup\tilde\Gamma$ with $\tilde\Gamma$ an admissible family of $t$ and in the later case that admissible families of $T$ are exactly the sets $\tilde\Gamma$, again with $\tilde\Gamma$ an admissible family of $t$.
   
   In both of these cases, we see that the contraction of $T$ with respect to $\Gamma$ cont$_\Gamma(T)$ is either cont$_{\tilde\Gamma}(t)$ or $B_+($cont$_{\tilde\Gamma}(t))$. In both cases we have $c_\Gamma=c_{\tilde\Gamma}$ and thus the first point of the proof holds.
   
   \item For a tree $T=B_+^a(T_1\cdots T_n)$ with $n\geq2$ the two cases have to be considered separately. First, let $\Gamma\subseteq V(T)$ be an admissible  family of $T$ which does not contains the root of $T$. Then since $T_\Gamma$ must be a rooted tree, we have that $\Gamma$ must contains vertices from at most one the the $T_i$. Thus in this case, $\Gamma$ is an admissible family of $T$ if, and only if, it is an admissible family of $T_i$ which does not contain all the vertices of $T_i$ (or otherwise the root has fertility $n-1$ in $T_{V(T)\setminus\Gamma}$). In this case have $\Gamma=\Gamma_i$ as claimed.
   
   For the case where $\Gamma\subseteq V(T)$ is an admissible  family of $T$ which does contains the root of $T$, we can conclude by observing that, according to the definition of admissible families, if $\Gamma$ is an admissible family, then so is $V(T)\setminus\Gamma$. We can also observe directly that for $T_{V(T)\setminus\Gamma}$ to be a rooted tree, $V(T)\setminus\Gamma$ must contains vertices from at most one of the $T_i$ since it does not contain the root, and that it must not contain all vertices of $T_i$ otherwise the root of $T$ has fertility one less in $T_{\Gamma}$ than in $T$. Thus we have in this case exactly $\Gamma=\{a\}\cup\Gamma_i\cup V(T_1\cdots\widehat{T_i}\cdots T_n)$ with $\Gamma_i$ an admissible family of $T_i$ which does not contain all the vertices of $T_i$.
   
   In both of these cases, to obtain $\text{cont}_\Gamma(T)$ we must contract each of the $T_j$ for $j\neq i$, in the first case as subset of $V(T)\setminus \Gamma$ and in the second case as subset of $\Gamma$. Thus we have
   \begin{equation*}
    \text{cont}_\Gamma(T)=B_+\left(\underbrace{\tun\cdots~\tun}_{n-1\text{ times}}\text{cont}_{\Gamma_i}(T_i)\right)
   \end{equation*}
   and therefore $c_\Gamma=nc_{\Gamma_i}$ as claimed.
  \end{enumerate}
 \end{proof}
 We can now state and prove a combinatorial formula for the dual coproduct.
 \begin{thm} \label{thm:coprod_adm}
  The coproduct $\Delta^*:\calF_\Omega^\circ\longrightarrow\calF_\Omega^\circ\otimes\calF_\Omega^\circ$ dual to the shuffle of rooted forests $\shuffle^T$ admits the following combinatorial description:
$$\Delta^*(T) = \sum_{\substack{ \Gamma\subseteq V(T) \text{ admissible}}} c_\Gamma  \; T_{V(T)\setminus\Gamma} \otimes T_{\Gamma}$$
for any rooted tree $T\in\calT_\Omega^\circ$ and for any $F=T_1\cdots T_n$ with $n\geq2$
$$\Delta^*(F) = F\otimes \emptyset + \emptyset\otimes F + \sum_{i= 1}^{n} \quad \sum_{\substack{\Gamma_i \subseteq V(T_i) \text{ admissible} \\ \emptyset \neq \Gamma_i \neq V(T_i)}} c_{\Gamma_i} 
\quad
\sum_{I\subseteq [n]_i}\alpha_{I,n} \;(T_i)_{V(T_i)\setminus\Gamma_i} T_I \otimes (T_i)_{\Gamma_i}T_{[n]_i\setminus I}$$
 with once again $\alpha_{I,n} := (|I|+1)(n - |I|)$.
 \end{thm}
 \begin{proof}
  We prove this result by induction on $N=|F|$. If $N=0$, the result trivially holds since $\emptyset\subseteq V(\emptyset)$ is the unique admissible family of the empty tree. \\
  
  For $N\in\N$, let us assume that the Theorem holds for all forests with $F$ with $|F|\leq N$ and let $T=B_+^a(F)$ be a rooted tree of $N+1$ vertices. Then by Equation \eqref{coproddualtree} of Theorem \ref{thm:coprod_dual_ind} we have 
  \begin{equation*}
   \Delta^*(T) = \underbrace{(B_+^a\otimes Id)(Id\otimes \Pi_{\mathcal{T}_\Omega})\Delta^*(F)}_{=:A} + \underbrace{(Id\otimes B_+^a)(\Pi_{\mathcal{T}_\Omega}\otimes Id)\Delta^*(F)}_{=:B}.
  \end{equation*}
  Then we have two cases to consider:
  \begin{enumerate}
   \item If $F=t\in\calT_\Omega^\circ$ then $|t|=N$ and we can use the induction hypothesis to obtain $\Delta^*(t)$. Since for any $\tilde\Gamma$ admissible family of $t$, $t_{V(t)\setminus\Gamma}$ and $t_\Gamma$ are rooted trees we have
   \begin{equation*}
    A = (B_+^a\otimes Id)\sum_{\substack{\tilde\Gamma\subseteq V(t) \\ \text{ admissible}}} c_{\tilde\Gamma}  \; t_{V(t)\setminus\tilde\Gamma} \otimes t_{\tilde\Gamma} = \sum_{\substack{\tilde\Gamma\subseteq V(t) \\\text{ admissible}}} c_{\tilde\Gamma}  \; B_+^a\left(t_{V(t)\setminus\tilde\Gamma}\right) \otimes t_{\tilde\Gamma}.
   \end{equation*}
   Using the first point of Lemma \ref{lem:adm_fam_ind} we see that the sum runs precisely on admissible families of $T$ that do not contain the root. We have 
   \begin{equation*}
    A =  \sum_{\substack{\Gamma\subseteq V(T) \text{ admissible}  \\ a\notin \Gamma}} c_{\Gamma}  \; B_+^a\left(t_{V(t)\setminus\Gamma}\right) \otimes t_{\Gamma} =  \sum_{\substack{\Gamma\subseteq V(T) \text{ admissible}  \\ a\notin \Gamma}} c_{\Gamma}  \; T_{V(T)\setminus\Gamma} \otimes T_{\Gamma}.
   \end{equation*}
   For the last equality, we have used that since $\Gamma\subseteq V(t)\Leftrightarrow a\notin\Gamma$ we have $t_\Gamma=T_\Gamma$ and $B_+^a\left(t_{V(t)\setminus\Gamma}\right)=T_{V(T)\setminus\Gamma}$.
   
   Using exactly the same argument we obtain 
   \begin{equation*}
    B =  \sum_{\substack{\Gamma\subseteq V(T) \text{ admissible}  \\ a\in \Gamma}} c_{\Gamma}  \; T_{V(T)\setminus\Gamma} \otimes T_{\Gamma}.
   \end{equation*}
   Therefore 
   \begin{equation*}
    \Delta^*(T) = \sum_{\substack{ \Gamma\subseteq V(T) \text{ admissible}}} c_\Gamma  \; T_{V(T)\setminus\Gamma} \otimes T_{\Gamma}
   \end{equation*}
   as claimed and we have proved the formula for $\Delta^*(T)$ at rank $|T|=N+1$ in the case $T=B_+^a(t)$.
   
   \item Now if $F=T_1\cdots T_n\in\calF_\Omega^\circ\setminus\calT_\Omega^\circ$, using the inductive formula \eqref{coproddualtree} of Theorem \ref{thm:coprod_dual_ind} and the induction hypothesis we find 
   \begin{align*}
    & A=(B_+^a\otimes Id)(Id\otimes \Pi_{\mathcal{T}_\Omega}) \\
    & \left(
    T_1\cdots T_n\otimes \emptyset + \emptyset\otimes T_1\cdots T_n + \sum_{i= 1}^{n} \sum_{\substack{\Gamma_i \subseteq V(T_i)\\ \text{ admissible} \\ \emptyset \neq \Gamma_i \neq V(T_i)}} c_{\Gamma_i} 
\sum_{I\subseteq [n]_i}\alpha_{I,n} \;(T_i)_{V(T_i)\setminus\Gamma_i} T_I \otimes (T_i)_{\Gamma_i}T_{[n]_i\setminus I}
\right)
   \end{align*}
and 
   \begin{align*}
    & B=(Id\otimes B_+^a)(\Pi_{\mathcal{T}_\Omega}\otimes Id) \\
    & \left(
    T_1\cdots T_n\otimes \emptyset + \emptyset\otimes T_1\cdots T_n + \sum_{i= 1}^{n}  \sum_{\substack{\Gamma_i \subseteq V(T_i)\\ \text{ admissible} \\ \emptyset \neq \Gamma_i \neq V(T_i)}} c_{\Gamma_i}
\sum_{I\subseteq [n]_i}\alpha_{I,n} \;(T_i)_{V(T_i)\setminus\Gamma_i} T_I \otimes (T_i)_{\Gamma_i}T_{[n]_i\setminus I}
\right).
   \end{align*}
   Since by hypothesis $T_1\cdots T_n$ is not a tree, the only on-vanishing terms in $A$ are $T_1\cdots T_n\otimes \emptyset$ and $I=[n]_i$. This implies $\alpha_{I,n}=n$ and we obtain
   \begin{equation*}
    A = T\otimes \emptyset + \sum_{i=1}^n\sum_{\substack{\Gamma_i \subseteq V(T_i)\\ \text{ admissible} \\ \emptyset \neq \Gamma_i \neq V(T_i)}} nc_{\Gamma_i}B_+^a\left((T_i)_{V(T_i)\setminus\Gamma_i}T_1\cdots \widehat{T_i}\cdots T_n\right) \otimes (T_i)_{\Gamma_i}
   \end{equation*}
   Using this time the second point of Lemma \ref{lem:adm_fam_ind} we find that the two sums are precisely summing over non-empty admissible families of $T$ that do not contain the root of $T$. Together with the term $T\otimes\emptyset$ we obtain
   \begin{equation*}
    A =  \sum_{\substack{\Gamma\subseteq V(T)\\ \text{ admissible}  \\ a\notin \Gamma}} c_{\Gamma}  B_+^a\left((T_i)_{V(T_i)\setminus\Gamma_i}T_1\cdots \widehat{T_i}\cdots T_n\right) \otimes (T_i)_{\Gamma} 
    =  \sum_{\substack{\Gamma\subseteq V(T)\\ \text{ admissible}  \\ a\notin \Gamma}} c_{\Gamma}   T_{V(T)\setminus\Gamma} \otimes T_{\Gamma}.
   \end{equation*}
   For the last equality, we have used as before that since $\Gamma\subseteq V(T_i)$ we have $(T_i)_\Gamma=T_\Gamma$ and $B_+^a\left((T_i)_{V(T_i)\setminus\Gamma_i}T_1\cdots \widehat{T_i}\cdots T_n\right)=T_{V(T)\setminus\Gamma}$.
   
   Using exactly the same argument we obtain 
   \begin{equation*}
    B =  \sum_{\substack{\Gamma\subseteq V(T) \text{ admissible}  \\ a\in \Gamma}} c_{\Gamma}  \; T_{V(T)\setminus\Gamma} \otimes T_{\Gamma}.
   \end{equation*}
   Therefore 
   \begin{equation*}
    \Delta^*(T) = \sum_{\substack{ \Gamma\subseteq V(T) \text{ admissible}}} c_\Gamma  \; T_{V(T)\setminus\Gamma} \otimes T_{\Gamma}
   \end{equation*}
   as claimed and we have proved the formula for $\Delta^*(T)$ at rank $|T|=N+1$ in the case $T=B_+^a(T_1\cdots T_n)$.
  \end{enumerate}
  This covers the two possible cases and we have proven the formula for $\Delta^*(T)$. \\
  
  We conclude the proof by noticing that for $F$ a non-connected rooted forest with $N+1$ vertices, the formula for $\Delta^*(F)$ follows directly from the formula for $\Delta^*(T)$ and the inductive formula of Theorem \ref{thm:coprod_dual_ind} for $\Delta^*(F)$. Therefore the Theorem holds for trees and forests with $N+1$ vertices and, by induction, holds for every trees and forests.
 \end{proof}

 \subsection{Dual bialgebra}
 
 We now endow $\calT_\Omega^\circ=\bigoplus_{n=0}^{+\infty}(\calT_\Omega)_n$ (where $(\calT_\Omega)_n$) is the vector space freely generated by rooted trees with $n$ vertices) with a bialgebra structure. Recall that for two trees $T$ and $T'$ and $v\in V(T)$, $T \underset{v}{\curvearrowleft} T'$ is the tree obtained by grafting the root of $T'$ as a new direct descendant of $v$.
 \begin{defi} \label{def:GL_like_products}
  For any $T$ and $T'$ in $\calT_\Omega^\circ$ set $\emptyset\lhd T'=T'$ if $T=\emptyset$ and, if $T\neq\emptyset$, set
  \begin{equation*}
   T\lhd T':=\sum_{l\in\ell(T)} T \underset{l}{\curvearrowleft} T'
  \end{equation*}
  where the sum runs over the set of leaves of $T$. We further define
  \begin{equation*}
    T \unlhd T' = \begin{cases}
        0, \quad \text{if} \quad |\ell(T)| > 1\\
        T \lhd T', \quad \text{otherwise.}
    \end{cases}
\end{equation*}
These two maps are then extended by linearaity to product $\lhd,\unlhd:\calT_\Omega^\circ\otimes\calT_\Omega^\circ\longrightarrow\calT_\Omega^\circ$.
 \end{defi}
 These products are non commutative and non associative. However, they are two restrictions of the Grossman-Larson product introduced in \cite{GL89,GL90,GL05} so it should be no surprise that they are pre-Lie products. Since this result is just in passing, we simply refer the reader to \cite{gerstenhaber1963cohomology} (although they were introduced earlier under a different name in \cite{vinberg1963theory}) for the definition of pre-Lie algebras, but they will be defined in the proof below.
 \begin{Prop} \label{prop:pre_Lie}
     $(\mathcal{T}^\circ_\Omega, \unlhd)$ is a right and a left pre-Lie algebra and $(\mathcal{T}^\circ_\Omega, \lhd)$ is a left pre-Lie algebra.
\end{Prop}
\begin{proof}
 We prove the two results separately. \\
 
    For $\unlhd$, let $t_1,t_2,t_3 \in \mathcal{T}_\Omega^\circ.$ We first prove that it is a right pre-Lie algebra, i.e. we want to prove that for any trees $t_1$, $t_2$ and $t_3$ we have:

    $$(t_1\unlhd t_2)\unlhd t_3 - t_1\unlhd (t_2\unlhd t_3) = (t_2\unlhd t_1)\unlhd t_3 - t_2\unlhd (t_1\unlhd t_3)$$
    We have four cases to consider.
    \begin{itemize}
        \item If $|\ell(t_1) = |\ell(t_2)| = |\ell(t_3)| = 1,$ then the product $\unlhd$ reduces to the concatenation of words. Thus it is associative and therefore pre-Lie.

        \item If $|\ell(t_1)| > 1$, then 
    $$t_1 \unlhd t_2 = t_1 \unlhd t_3 = t_1\unlhd (t_2 \unlhd t_3) = 0$$
    and $|\ell(t_2 \unlhd t_1)| \geq |\ell(t_1)|$ therefore $$(t_2 \unlhd t_1)\unlhd t_3 = 0$$
    and the result follows.
    \item The case $|\ell(t_2)| > 1$ is similar to the previous one.

    \item The last case $l(t_3) > 1$ follows from the fact that 
    $$(t_1\unlhd t_2)\unlhd t_3 \; \text{and} \; t_1\unlhd (t_2\unlhd t_3)$$
    are the trees obtained by grafting $t_3$ on top of $t_2$ and $t_2$ on top of $t_1$ and hence are equal. The same argument follows for the other side of the pre-Lie identity.
    \end{itemize}
    The fact that $\unlhd$ is also a left pre-Lie algebra is proven in the exact same way and we omit the details here.
    
    For $\lhd$, we want to show $(t_1\lhd t_2)\lhd t_3 - t_1\lhd (t_2\lhd t_3) = (t_1\lhd t_3)\lhd t_2 - t_1\lhd (t_3\lhd t_2)$. This can be checked through a tedious verification of the arguments in the sums of each sides, but we give instead a simpler, more graphical argument. Notice that this is the usual argument to show that the Grossman-Larson product is pre-Lie.
    
    First, observe that $(t_1\lhd t_2)\lhd t_3$ is a linear combination of two types of trees. Trees where $t_3$ is grafted on a leaf of $t_2$ which is itself grafted on a leaf of $t_1$; and trees where $t_3$ and $t_2$ are both grafted on leafs of $t_1$. But $t_1\lhd (t_2\lhd t_3)$ contains precisely the trees where trees where $t_3$ is grafted on a leaf of $t_2$ which is itself grafted on a leaf of $t_1$. Thus $(t_1\lhd t_2)\lhd t_3 - t_1\lhd (t_2\lhd t_3)$ contains only the trees where $t_3$ and $t_2$ are both grafted on leafs of $t_1$.
    
    On the other hand, with the same argument but exchanging $t_3$ and $t_2$ we see that $(t_1\lhd t_3)\lhd t_2 - t_1\lhd (t_3\lhd t_2)$ contains also the trees where $t_3$ and $t_2$ are both grafted on leafs of $t_1$. Therefore the two sides of the equality are equal and $\lhd$ is left pre-Lie.

\end{proof}
The next result is the main one of this subsection. It explains why the coproduct of Definition \ref{def:shuffle_coprod} is canonical, and it also implies that it is co-right pre-Lie.
\begin{thm} \label{thm:dual_bialg}
 $(\calT_\Omega^\circ,\unlhd,\Delta^*)$ is the bialgebra dual to $(\mathcal{T}_\Omega,\shuffle^T,\Delta)$ as bialgebra.
\end{thm}
\begin{proof}
 One can directly check that the coproduct $\Delta^*$ is a morphism for $\unlhd$, i.e. that 
 \begin{equation*}
        \Delta^*(T_1\unlhd T_2) = (\unlhd \otimes \unlhd)\circ \tau_{23}\circ (\Delta^* \otimes \Delta^* (T_1 \otimes T_2))
    \end{equation*}
    for any rooted trees $T_1$ and $T_2$. However, since by Theorem \ref{thm:shuffle_bialg} $(\mathcal{T}_\Omega,\shuffle^T,\Delta)$ is a bialgebra, it is enough to prove that $\unlhd$ is the product dual to the coproduct $\Delta$, i.e. that, for any trees $(T_1,T_2)\in(\calT_\Omega^\circ)^2$ we have 
    \begin{equation} \label{eq:unlhd_bialg}
     \langle T_1\unlhd T_2,T\rangle=\langle T_1\otimes T_2,\Delta(T)\rangle~ \forall T\in\calT_\Omega.
    \end{equation}
    Let us assume that $T_1$, $T_2$ and $T$ are rooted trees. The general cases of $T_1$, $T_2$ and $T$ to be linear combinations is obtained by linearity. We have three cases to consider.
    
    If $T_1=\emptyset$ we have 
    \begin{equation*}
     \langle T_1\unlhd T_2,T\rangle = \langle T_2,T\rangle=\delta_{T_2T}
    \end{equation*}
    with $\delta$ the Kronecker symbol. On the other hand, we then have $\langle T_1\otimes T_2,\Delta(T)\rangle=\langle \emptyset\otimes T_2,\Delta(T)\rangle$. The only possible  non-vanishing term in this expression is $\emptyset\otimes T$. Thus we also obtain 
    \begin{equation*}
     \langle T_1\otimes T_2,\Delta(T)\rangle =\delta_{T_2T}
    \end{equation*}
    and Equation \eqref{eq:unlhd_bialg} holds in the case $T_1=\emptyset$.
    
    If $|\ell(T_1)|=1$ then $T_1$ is a linear tree with $n\in\N^*$ vertices. Then $T_1\unlhd T_2$ is the tree obtained by grafting $T_2$ on the unique leaf of $T_1$. On the other hand,
     in this case, we have by definition of $\Delta$ that $\langle T_1\otimes T_2,\Delta(T)\rangle$ is non-zero if,and only if,  the trunk of $T$ to be of length at least $n$, and the tree $T'$ obtained by removing the first $n$ vertices of the trunk of $T$ to be $T_2$. This is equivalent to saying that we have $T=T_1\underset{l}{\curvearrowleft} T_2=T_1\unlhd T_2$, with $l$ the unique leaf of $T_1$. Therefore, in this case, Equation \eqref{eq:unlhd_bialg} holds as well.
     
     Finally, if $|\ell(T_1)\geq2$ then $\langle T_1\unlhd T_2,T\rangle =0$ for all $T\in\calT_\Omega$. On the other hand, by definition of $\Delta$ (Definition \ref{def:shuffle_coprod}) we have for any $T\in\calT_\Omega$
     \begin{equation*}
      \Delta(T)=\sum_{n=0}^{N_T-1}\ell_n\otimes T''
     \end{equation*}
     with $N_T$ the number of vertices in the trunk of $T$ (thus $N_T\geq1$ since $T$ is a tree) and $\ell_n$ the linear tree with $n$ vertices. The terms in the left of these tensors have always at most one leaf. Thus, if $T_1$ has two or more leaves we have $\langle T_1\otimes T_2,\Delta(T)\rangle=0$ for any $T\in\calT_\Omega$. Thus Equation \eqref{eq:unlhd_bialg} holds also in this case, and the Theorem is proven.

\end{proof}

\subsection{Dual primitive trees}

Let us now caracterize and enumerate the rooted trees that are primitive for the dual coproduct $\Delta^*$, i.e. the trees such that 
\begin{equation*}
 \Delta^*(T)=T\otimes\emptyset+\emptyset\otimes T.
\end{equation*}
\begin{remark}
 Notice that this is not the same as asking what are the primitive elements of $\Delta^*$. Linear combination of non primitive trees could a priori be primitive. Finding all primitive elements is much harder but less interesting for our purpose, since we are in the dual and since it is not clear that we have a Milnor-Moore  theorem (\cite{MM65}, see also \cite{Ca06} and \cite{Pa94}) for non associative Hopf algebras. 
\end{remark}
The caracterization of primitive rooted trees eventually reduces to the following two technical lemmas.
\begin{Lemme}\label{lemshufflegraft}
   Let $T$ and $T'$ be two non-empty rooted trees, let $v \in V(T)$ be a leaf of $T$ and $T' = B_+^{v'}(F').$ Then 
   \begin{equation*}
    \langle\Delta^*(T \underset{v}{\curvearrowleft} T'),T\otimes T'\rangle\neq0.
   \end{equation*}
\end{Lemme}
\begin{proof}
    We prove this by induction on $N=|T|+|T'|$. If $N=2$, then $T$ and $T'$ both have only one vertex since both are non-empty, and $T \underset{v}{\curvearrowleft} T'$ is a linear tree with two vertices. Then the result follows from a direct computation.
    
    Suppose that for some $N\geq3$ the statement holds for any rooted trees $t$ and $t'$ such that $|t|+|t'| < N$. Let $T$ and $T'$ be two rooted trees such that $|T|+|T'| = N.$ If $|T|=1$ then as before a direct computation with the inductive formula \eqref{coproddualtree} for $\Delta^*(T\underset{v}{\curvearrowleft} T')$.

    If $|T|\geq2$, write $T = B_+^a(T_1\dots T_k)$ and $T' = B_+^b(T'_1\dots T'_l)$ and consider $v$ a leaf of $T.$  Then we have
    $$T \shuffle^T T' = B_+^a\left((T_1\dots T_k) \shuffle^T T'\right) + B_+^b\left(T \shuffle^T (T'_1\dots T'_l)\right).$$
    Since $|T|\geq2$, we have that $v$ is a leaf one of $T_i,$ for one $i\in \{1,\dots,k\}.$ Then, since the shuffle $T_i\shuffle ^T T'$ appears in $(T_1\dots T_k) \shuffle^T T',$ and $|T_i|+|T'|<N,$ we can use the induction hypothesis to conclude that $T_i\otimes T'$ appears in the coproduct $\Delta^*(T_i\underset{v}{\curvearrowleft} T')$. Going to the dual, this is equivalent to saying that $T_i\underset{v}{\curvearrowleft} T'$ appears in $T_i\shuffle^T T'$. Therefore $T\underset{v}{\curvearrowleft} T'$ appears in $T\shuffle^T T'$. Again going to the dual, this is equivalent to saying that $T\otimes T'$ appears in $\Delta^*(T\underset{v}{\curvearrowleft} T')$, which is what we wanted to show. 
\end{proof}
\begin{Lemme}\label{lemshuffert}
    Let $T,T'$ be non-empty rooted trees and $t\in\calT_\Omega^\circ$ such that $\langle t, T\shuffle T'\rangle\neq0$. Then there exists $ v \in V(t)$ such that $\text{fert}(t) = 1.$
\end{Lemme}
\begin{proof}
    Let $T,T'$ be non-empty rooted trees and let us prove once again this result by induction on $N=|T|+|T'|.$ For $|T|+|T'| = 2$ we have $|T| = |T'| = 1,$. Thus we can write $T = \tdun{a}$ and $T' = \tdun{b}$ for some $(a,b)\in\Omega^2$. Then 
        $$T \shuffle^T T' = \tddeux{a}{b} + \tddeux{b}{a}$$
        and the result holds.

    Assuming that for some $N\geq3$ the result holds for any $t,t'$ with $|t|+|t'| < N,$ let $T,T'$ such that $|T|+|T'| = N.$ We can write $T = B_+^a(F)$ and $T' = B_+^b(F'),$ in this case
    $$T\shuffle^T T' = B_+^a\left(F\shuffle^T T'\right) + B_+^b\left(T\shuffle^T F'\right).$$

    In both expressions $F\shuffle^T T'$ and $T\shuffle^T F'$ we have shuffles of trees with less than N nodes. Therefore, by the induction hypothesis, a tree of each of the forests appearing in these product has vertility exactly one. Therefore it always exists a vertex in both $B_+^a(F\shuffle T')$ and $B_+^b(T\shuffle F')$ that has fertility exactly one and the lemma is proven by induction.
\end{proof}
We can now fully characterize the primitive trees for the dual coproduct $\Delta^*$.
\begin{thm}  \label{thm:description_prim_trees}
    A tree is primitive for $\Delta^*$ if, and only if, each of its vertex except the leafs has fertility at least two.
\end{thm}
\begin{proof}
    $(\implies)$ Let $T$ be a tree and $v \in V(T).$ Suppose $fert(v) = 1$ and let $v'$ be the direct descendent of $v.$ We can write $T = \tilde{T} \underset{v}{\curvearrowleft} \tilde{T'}$ where $\tilde{T'}$ has $v'$ as its root and $v$ is a leaf of $\tilde{T}.$ Therefore, from the Lemma \ref{lemshufflegraft}, $T$ is not primitive.

    $(\impliedby)$ Let $T\in\calT_\Omega^\circ$ be a tree and suppose that it's not primitive for $\Delta^*.$ Therefore, there exists two rooted trees $t_1$ and $t_2$ such that $\emptyset \neq t_1 \neq T,$ $\emptyset \neq t_2\neq T$ and 
    \begin{equation*}
     \langle T,t_1\shuffle^T t_2\rangle = \langle \Delta^*(T),t_1\otimes t_2\rangle\neq 0.
    \end{equation*}
    By Lemma \ref{lemshuffert}, each such tree in the set has at least one vertex with fertility one, hence the result holds.
\end{proof}
This Theorem allows us to deduce an inductive formula for the number of primitive trees with $n$ vertices. Let $p_n$ be the number of primitive trees for $\Delta^*$, not counting linear combinations. As stated above, $p_n$ is a lower bound for the dimension of the space of primitive elements for $\Delta^*$.
\begin{Prop} \label{prop:prim_trees}
    $p_n$ is given by the recursive formula $p_0=0$, $p_1=1$, $p_2=0$ and, for $n\geq2$:
    \begin{equation} \label{eq:induction_primitive}
        p_{n+1}=\sum_{k=2}^n \; \sum_{\substack{n_1\geq\cdots\geq n_k\geq1 \\ n_1+\cdots +n_k=n}}\; \prod P_{n_i}
    \end{equation}
    where $P_{n_i} = \binom{p_{n_i} + k_{n_i} - 1}{k_{n_i}}$ and $k_{n_i}$ is the number of times $n_i$ appears in the $k$-tuple $(n_1,\dots,n_k)$,and where the product is over each distinct $n_i$.
\end{Prop}
\begin{proof}
     The initial values of $p_n$ are computed by enumeration of the primitive trees. For the inductive formula, let $n\geq2$ and $T=B_+(T_1\cdots T_k)$ be a primitive tree with $n+1$ vertices. If $k=1$, the root has only one direct descendant, which is in contradiction with the previous characterisation of primitive trees. Thus we must have $k\geq2$. Since $T$ has $n+1$ vertices and each of the $T_i$ has at least one vertex, we must also have $k\leq n$.

    Now set $n_i:=|V(T_i)|$. For $T$ to have $n+1$ vertices, we must have $n_1+\cdots+n_k=n$. Since a permutation of $T_i$ does not change $T$, we have that $(n_1,\cdots,n_k)$ must be a partition of the integer $n$ (i.e. a finite unordered sequence whose terms sum to $n$). Up to a relabelling, these terms can be put in a decreasing order. Thus we obtain the summand of the second sum sign in Equation \eqref{eq:induction_primitive}.

    Notice that for each $T_i$, each of the internal vertices of $T_i$ is an internal vertex of $T$, and thus has at least two descendants since $T$ is primitive. Then, for each $n_i,$ consider $k_{n_i}$ as the number of times $n_i$ appears in the partition $n.$ The number of different ways we can arrange $p_{n_i}$ trees on $k_{n_i}$ spots is exactly the number of possible $k_{n_i}$-combinations with repetitions on the set $p_{n_i}$ primitive rooted trees with $n_i$ vertices. It is a classic result of combinatorics (see for example \cite{morgado1991analise}) that this number is $$\binom{p_{n_i} + k_{n_i} -1}{k_{n_i}}.$$
    The product over all those permutations for each $n_i$ gives us the product on the formula.
    
\end{proof}

With this formula one can compute the following sequence of the number PT of primitive trees with V$\in\N$ vertices.

\begin{center}
\begin{tabular}{ |c|c|c|c|c|c|c|c|c|c|c|c|c|c| } 
 \hline
 V & 0 & 1 & 2 & 3 & 4 & 5 & 6 & 7 & 8 & 9 & 10 & 11 & 12 \\ 
 \hline 
 PT & 0 & 1 & 0 & 1 & 1 & 2 & 3 & 6 & 10 & 19 & 35 & 67 & 127 \\ 
 \hline
\end{tabular}
\end{center}

\begin{center}
\begin{tabular}{ |c|c|c|c|c|c|c|c|c|c|c|c| } 
 \hline
 V & 13 & 14 & 15 & 16 & 17 & 18 & 19 & 20 & 21 & 22 & 23  \\ 
 \hline 
 PT & 248 & 482 & 952 & 1885 & 3765 & 7546 & 15221 & 30802 & 62620 & 127702 & 261335  \\ 
 \hline
\end{tabular}
\end{center}
This sequence does not seem to appear on the Online Encyclopedia of Online Sequences.

\section{Rooted trees and Rota-Baxter algebras} \label{sec:three}

In this section, $A$ will always denote an associative and commutative algebra over a field $\K$ whose product we omit for readability.

\subsection{Prerequisites on Rota-Baxter algebras}

Rota-Baxter algebras (RBA for short) were introduced  \cite{baxter1960analytic} and further developped in \cite{rota1969baxter} and \cite{cartier1972structure} among others.
We recall the needed definition and results of the theory.
\begin{defi} \label{def:RBA_cat}
 \begin{enumerate}
  \item \label{superdef}
  Let $\lambda \in \K.$  A \textbf{Rota-Baxter algebra of weight $\lambda$} (RBA of weight $\lambda$) is a pair $(R,P)$ consisting of an algebra $R$ and a linear operator $P: R\to R$ that satisfies the equation
		$$P(x)P(y) = P(xP(y)) + P(P(x)y) + \lambda P(xy), \quad \forall x,y \in R.$$
  \item \label{Rota-Baxter morphism}
		Let $(R_1,P_1)$ and $(R_2,P_2)$ be two Rota-Baxter algebras of the same weight $\lambda.$ A \textbf{Rota-Baxter homomorphism} is an algebra homomorphism $\bar{f}: (R_1,P_1) \to (R_2,P_2)$ such that 
		$$\bar{f} \circ P_1 = P_2 \circ \bar{f}.$$
 \end{enumerate}
\end{defi}
Let $A$ be an algebra. Recall that a free Rota-Baxter algebra of weight $\lambda$ on $A$ is an RBA $(R,P)$ of weight $\lambda$ and an algebra homomorphism $\phi : A \to R,$ such that for any RBA $(R',P')$ of weight $\lambda$ and algebra homomorphism $f : A \to R'$ there is a unique Rota-Baxter algebra morphism $\bar{f}: (R,P) \to (R',P')$ such that $f = \bar{f} \circ \phi$. Free RBA were described in \cite{GUO2000117} (see also \cite{Guo}). We now recall this construction.

For an algebra $A$, recall that $\calW_A$ is the vector space freely generated by words written in the alphabet $A$. We denote the $\emptyset$ the empty word which is an element of $\calW_A$ and by $\sqcup$ the concatenation product on $\calW_A$:
\begin{equation*}
 (a_1\cdots a_p)\sqcup(b_1\cdots b_q)=(a_1\cdots a_pb_1\cdots b_q).
\end{equation*}

We further endow $\calW_A$ with the product $\diamond_\lambda$ defined by:
$$a\diamond_\lambda b = (a_0.b_0)\sqcup(a'\shuffle_\lambda b')$$	
where $a = a_0 \sqcup a', \; b = b_0 \sqcup b'$. As before, $a_0.b_0\in A$ is the multiplication in $A$ of $a_0$ and $b_0$.

Further set $P_A : W_A \to W_A$ to be
$$P_A(w) = (1_A) \sqcup w$$
and consider $j_A: A \to \calW_A$ the canonical inclusion map. The result of \cite{GUO2000117,Guo} then reads:
\begin{thm}
 $(\calW_A,P_A)$ is a Rota-Baxter algebra of weight $\lambda$ for the product $\diamond_\lambda$. Together with the map $j_A,$ it is the free commutative Rota-Baxter algebra on A of weight $\lambda.$ 
\end{thm}

\subsection{Diamond products on forests}

$\lambda$-shuffles on forests were already introduced in Subsection \ref{subsec:shuffle_coalg}. Thus one could built $\diamond_\lambda$ products on forests from these $\lambda$-shuffles on forests, but this construction does not have so much good properties. In order to take into account the concatenation product of rooted forests (which is very different from the concatenation of words, since the former is commutative while the later is not) we need modify the $\lambda$-shuffles of rooted forests. Recall that for a set $\Omega$, we set $\calF_\Omega^+:=\calF_\Omega\setminus\{\emptyset\}$ to be the set of non-empty forests.
\begin{defi} \label{def:new_shuffle}
 Let $A$ be a commutative algebra over $\K$ and set $\lambda\in\K$. We define recursively the product $\diamond_\lambda$ on $\calF_A^+$ et $\ast_\lambda$ on $\calF_A$.
 \begin{itemize}
  \item Set $\emptyset\ast_\lambda\emptyset:=\emptyset$
 \end{itemize}
 Assume now that for $N\geq1$, $F\diamond_\lambda F'$ has been defined for any non empty forests $F$ and $F'$ such that $|F|+|F'|\leq N$, and that $F\ast_\lambda F'$ has been defined for any forests $F$ and $F'$ such that $|F|+|F'|\leq N$. Let $f$ and $f'$ be two forests such that $|f|+|f'|=N+1$.
 \begin{itemize}
  \item If $f'=\emptyset$  set $f \ast \emptyset = \emptyset\ast f = f$.
  \item If $f = T_1 \cdots T_k$ and $f' = t_1 \cdots t_n$ are both non-empty set 
$$f\diamond_\lambda f':= \frac{1}{kn} \sum_{i = 1}^{k} \sum_{j = 1}^{n} \big((T_i \diamond_\lambda t_j) T_1 \dots \widehat{T_i}\dots T_k t_1 \dots \widehat{t_j} \dots t_n\big).$$
  \item Again if $f = T_1 \cdots T_k$ and $f' = t_1 \cdots t_n$ are both non-empty set 
\begin{equation*}
 f\ast_\lambda f' := \dfrac{1}{k}\sum_{i = 1}^{k}B_+^{a_i}(f_i\ast_\lambda f')T_1 \dots \widehat{T_i}\dots T_k + \dfrac{1}{n}\sum_{j = 1}^{n}B_+^{a'_j}(f\ast_\lambda f'_j)t_1 \dots \widehat{t_j}\dots t_n + \lambda f\diamond_\lambda f'
\end{equation*}
where we used $T_i = B_+^{a_i}(f_i)$ and $t_j = B_+^{a'_j}(f'_j).$
  \item If $f$ and $f'$ are both non-empty trees $f=t=B_+^a(F)$ and $f'=t'=B_+^{a'}(F')$ we set 
 \begin{equation*}
  B_+^a(F) \diamond_\lambda B_+^{a'}(F') := B_+^{aa'}(F \ast_\lambda F')
 \end{equation*}
  with $aa'$ being the product of $a$ and $a'$ in $A$.
 \end{itemize}
\end{defi}
\begin{remark}
 Notice that for any trees $t,t'$ in $\calT_A$, $T\ast_\lambda T'=T\shuffle_\lambda^T T'$ (see Subsection \ref{subsec:shuffle_coalg} below for the definition of the $\lambda$-shuffle products of rooted trees and forests). However, the two products do not coincide on non-connected forests. In this study, $\ast_\lambda$ only appears as an intermediate product for the definition of $\diamond_\lambda$ and we shall not study it further here.
\end{remark}
Before stating the main property of the $\diamond_\lambda$, let us write some examples of this product.
\begin{Eg}
 Let $A$ be a commutative algebra, and $a$, $b$, $c$, $d$ and $e$ be elements of $A$. Then
 \begin{align*}
  & \tdun{a}\diamond_0 \tdun{b} = \tdun{ab},\quad \tdun{a} \diamond \tdun{b}\tdun{c} = \frac{1}{2}(\tdun{ab}\tdun{c} + \tdun{ac}\tdun{b}), \\ 
  & \tddeux{a}{b}\diamond_\lambda \tdtroisun{c}{d}{e} = \tdquatrequatre{ac}{b}{d}{e} + \frac{1}{2}\big(\tdquatredeux{ac}{e}{d}{b} + \tdquatredeux{ac}{d}{e}{b}\big) + \frac{\lambda}{2}\big(\tdtroisun{ac}{bd}{e} + \tdtroisun{ac}{be}{d}\big).
 \end{align*}
\end{Eg}
The main result of this Subsection is then
\begin{thm} \label{thm:quasi_univ_prop}
	Let $A$ be a unitary commutative algebra over $\K$ and let $\lambda\in\K$. Then $(\calF_A,\diamond_\lambda, B_+^{1_A})$ is a Rota-Baxter algebra of weight $\lambda$. Furthermore, for any commutative associative Rota-Baxter algebra $(R,P)$ of weight $\lambda$ and any algebra homomorphism $\varphi: A\to R$ there is a Rota-Baxter algebra homomorphism $\bar{\varphi}: \calF^+_A \to R$ such that the following diagram commutes
 \begin{equation}\label{diagramdiam}
\begin{tikzcd}
A \arrow[rd, "\varphi"] \arrow[r, "j_A"] & {(\calF^+_{A}, B_+^{1_{A}})} \arrow[d, "\bar{\varphi}"] \\
                                               & {(R,P)}                                          
\end{tikzcd}
\end{equation}
    where $j_A: A \to \calT_A$ is the natural embedding.

    Furthermore, $\bar{\varphi}$ is the unique such Rota-Baxter algebra morphism with this property which is also a morphism for the concatenation product of trees.
\end{thm}
\begin{proof}
 Let us first prove that $(\calF^+_A,\diamond_\lambda, B_+^{1_A})$ is a Rota-Baxter algebra of weight $\lambda$. Let $f = T_1 \cdots T_k$ and $f' = t_1 \cdots t_n$ be two forests. Since the $T_i$ and $t_j$ are non empty we can write  $T_i=B_+^{a_i}(f_i)$ and $t_j=B_+^{a'_j}(f'_j)$. With a small abuse of notation we also write $\psi_i(f) = T_1 \cdots \hat{T_i} \cdots T_k$ and $\psi_j(f') = t_1 \cdots \hat{t_j} \cdots t_k$ (where the hat means that the tree is omitted). Then
 \begin{align*}
        B_+^{1_A}(f)\diamond_\lambda B_+^{1_A}(f') &= B_+^{1_A}(f\ast_\lambda f') \\
        &= B_+^{1_A}\left(\frac{1}{k}\sum_{i = 1}^{k}B_+^{a_i}(f_i\ast_\lambda f')\psi_i(f) + \frac{1}{n}\sum_{j = 1}^{n}B_+^{a'_j}(f\ast_\lambda f'_j)\psi_j(f') + \lambda f\diamond_\lambda f' \right)\\
        &= B_+^{1_A}\left(\frac{1}{k}\sum_{i = 1}^{k}\left(T_i\diamond_\lambda B_+^{1_A}(f')\right) \psi_i(f) + \frac{1}{n}\sum_{j = 1}^{n}\left(B_+^{1_A}(f)\diamond_\lambda t_j\right) \psi_j(f') + \lambda f\diamond_\lambda f' \right)\\
        &= B_+^{1_A}\left(f\diamond_\lambda B_+^{1_A}(f') + B_+^{1_A}(f)\diamond_\lambda f' + \lambda f\diamond_\lambda f'\right).
    \end{align*}
 Thus $(\calF^+_A,\diamond_\lambda, B_+^{1_A})$ is a Rota-Baxter algebra of weight $\lambda$ as claimed. \\
 
 To show the property \ref{diagramdiam}, let $(R,P)$ be a Rota-Baxter algebra of weight $\lambda$ and $\varphi : A \to R$ be an algebra homomorphism. Define $\bar{\varphi}: (F_A,B_+^{1_A}) \to (R,P)$ by induction on $|F|$.
    When $n = 1$ we have $F = \tdun{a}$ and
    $$\bar{\varphi}(\tdun{a}) = \bar{\varphi}(j_A(a)) := \varphi(a).$$
 Assume now that for $N\geq2$ fixed, $\bar{\varphi}(f)$ has been defined for all forest such that $|f| < N$ and let $|F| = N.$ We have two cases
    \begin{itemize}
        \item If $F = B_+^a(F')$ with $|F'| = N-1$ we set 
        \begin{equation*}
         \bar{\varphi}(F) = \bar{\varphi}\left(B_+^a(F')\right):=\varphi(a)P(\bar{\varphi}(F')).
        \end{equation*}
        \item If $F = T_1\dots T_n$ with $n\geq2$ we set 
        $$\bar{\varphi}(F) := \bar{\varphi}(T_1)\dots \bar{\varphi}(T_n).$$
    \end{itemize}
    Let us now show that $\bar{\varphi}$ is indeed a Rota-Baxter homomorphism. For any $F\in\calF_A^+$ we have
    \begin{equation*}
     \bar{\varphi}\left(B_+^{1_A}(F)\right)=\varphi(1_A)P\left(\bar{\varphi(F)}\right) = P\left(\bar{\varphi(F)}\right)
    \end{equation*}
    where we have used the definition of $\bar{\varphi}$ on trees and the fact that $\varphi$ is an algebra morphism and therefore $\varphi(1_A)=1_R$. Thus $\bar{\varphi}\circ B_+^{1_A} = P\circ \bar{\varphi}$ as expected.
    
    We should now show that $\bar\varphi$ is an algebra morphism. For any non-empty forests $F_1$ and $F_2$ we prove $\bar{\varphi}(F_1\diamond F_2) = \bar{\varphi}(F_1)\bar{\varphi}(F_2)$ by induction on $|F_1|+|F_2|$. When $|F_1|+|F_2| = 2$ then $F_1 = \tdun{a}$ and $F_2 = \tdun{b}$ then
    $$\bar{\varphi}(\tdun{a}\diamond_\lambda \tdun{b}) = \bar{\varphi}(\tdun{ab}) = \varphi(ab) = \varphi(a)\varphi(b) = \bar{\varphi}(\tdun{a})\bar{\varphi}(\tdun{b})$$
    where we have used the definition of $\bar\varphi$ and the fact that $\varphi$ is an algebra morphism.
    
    Now assume that for $N\geq3$ the result holds for $|F_1|+|F_2| < N$ and let $F_1,F_2$ be non-empty forests such that $|F_1|+|F_2| = N.$ We have two cases to consider:
    \begin{itemize}
     \item If $F_1 = T_1\dots T_k$ and $F_2 = t_1\cdots t_n$ we have
    \begin{align*}
        \bar{\varphi}(F_1\diamond_\lambda F_2) &= \bar{\varphi}\left(\frac{1}{kn} \sum_{i = 1}^{k} \sum_{j = 1}^{n} \big((T_i \diamond_\lambda t_j) T_1 \dots \widehat{T_i}\dots T_k t_1 \dots \widehat{t_j} \dots t_n\big)\right) \quad\text{by definition of }\diamond_\lambda\\
        &= \frac{1}{kn}\sum_{i = 1}^{k} \sum_{j = 1}^{n}\bar{\varphi}(T_i\diamond_\lambda t_j)\bar{\varphi}(T_1)\dots \widehat{\bar{\varphi}(T_i)}\dots \bar{\varphi}(T_k)\bar{\varphi}(t_1)\dots \widehat{\bar{\varphi}(t_j)}\dots \bar{\varphi}(t_n)\quad\text{by definition of }\bar\varphi\\
        &= \frac{1}{kn}\sum_{i = 1}^{k} \sum_{j = 1}^{n}\bar{\varphi}(T_1)\dots  \bar{\varphi}(T_k)\bar{\varphi}(t_1)\dots \bar{\varphi}(t_n)\quad\text{by the induction the hypothesis}\\
        &= \bar{\varphi}(F_1)\bar{\varphi}(F_2) \quad\text{by definition of }\bar\varphi.
    \end{align*}
     \item If however $F_1 = B_+^a(F'_1)$ and $F_2 = B_+^b(F'_2)$, using the definitions of $\diamond_\lambda$ and $\ast_\lambda$ we can write 
     \begin{equation*}
      \tdun{ab}\diamond_\lambda\left(B_+^{1_A}(F'_1)\diamond_\lambda B_+^{1_A}(F'_2)\right)=\tdun{ab}\diamond_\lambda\left(B_+^{1_A}(F_1'\ast_\lambda F_2')\right) = B_+^{ab}(F_1'\ast_\lambda F_2') = F_1\diamond_\lambda F_2.
     \end{equation*}
 We then have
\begin{align*}
		\bar{\varphi}(F_1\diamond_\lambda F_2) & = \bar{\varphi}\Big(\tdun{ab}\diamond_\lambda\left(B_+^{1_A}(F'_1)\diamond_\lambda B_+^{1_A}(F'_2)\right)\Big)\\
		& = \bar{\varphi}\Big(\tdun{ab}\diamond_\lambda B_+^{1_A}\left(F'_1\diamond_\lambda B_+^{1_A}(F'_2) + B_+^{1_A}(F'_1)\diamond_\lambda F'_2 + \lambda F'_1\diamond_\lambda F'_2\right)\Big) \\
        & \hspace{8cm}
        \llcorner\text{since }B^{1_A}_+\text{ is a RBO of weight }\lambda\\
		& = \bar{\varphi}\Big( B_+^{ab}\left(F'_1\diamond_\lambda B_+^{1_A}(F'_2) + B_+^{1_A}(F'_1)\diamond_\lambda F'_2 + \lambda F'_1\diamond_\lambda F'_2\right)\Big) \\
		&= \varphi(ab)\bar{\varphi}\Big(B_+^{1_A}\left(F'_1\diamond_\lambda B_+^{1_A}(F'_2) + B_+^{1_A}(F'_1)\diamond_\lambda F'_2 + \lambda F'_1\diamond_\lambda F'_2\right)\Big)\quad\text{by definition of }\bar\varphi\\
		&= \varphi(ab)P\Big(\bar{\varphi}\left(F'_1\diamond_\lambda B_+^{1_A}(F'_2) + B_+^{1_A}(F'_1)\diamond_\lambda F'_2 + \lambda F'_1\diamond_\lambda F'_2\right)\Big) \quad\text{since }\bar\varphi\circ B_+^{1_A}=P\circ \bar\varphi\\
		&= \varphi(ab)P\Big(\bar{\varphi}(F'_1)\bar{\varphi}\left(B_+^{1_A}(F'_2)\right) + \bar{\varphi}\left(B_+^{1_A}(F'_1)\right)\bar{\varphi}(F'_2) + \lambda\bar{\varphi}(F'_1)\bar{\varphi}(F'_2)\Big) \\
		&\hspace{8cm}\llcorner\text{by the induction hypothesis}\\
		&= \varphi(ab)P\Big(\bar{\varphi}(F'_1)P\left(\bar{\varphi}(F'_2)\right) + P\left(\bar{\varphi}(F'_1)\right)\bar{\varphi}(F'_2) + \lambda\bar{\varphi}(F'_1)\bar{\varphi}(F'_2)\Big) \quad\text{since }\bar{\varphi}\circ B_+^{1_A} = P\circ \bar{\varphi}\\
        & = \varphi(ab)P\Big(\bar{\varphi}(F'_1)\Big)P\Big(\bar{\varphi}(F'_2)\Big) \quad \text{since }P\text{ is a RBO of weight }\lambda\\
        & = \varphi(a)\varphi(b)P\Big(\bar{\varphi}(F'_1)\Big)P\Big(\bar{\varphi}(F'_2)\Big) \quad \text{since }\varphi\text{ is an algebra morphism}\\
		&= \bar{\varphi}(F_1)\bar{\varphi}(F_2)
	\end{align*}
	by definition of $\bar\varphi$ and commutativity of $R$.
    \end{itemize}
Therefore we $\bar\varphi$ is an algebra morphism. \\

Finally, notice that by construction, $\bar\varphi$ is an algebra morphism for the concatenation product of trees. Let us show by contradiction that it is the only morphism for this product that makes Diagram \ref{diagramdiam} commutes. Let $\psi :(\cal{F}^+_{A},B_+^{1_{A}}) \to (R,P)$ be a Rota-Baxter algebra morphism such that Diagram \ref{diagramdiam} and that is a homomorphism for the concatenation of trees and suppose $\psi \neq \bar{\varphi}$.

So it exists $f\in \calF^+_{A}$ such that $\psi(f)\neq \bar{\varphi}(f)$. Choose such an $f$ with the less possible vertices and set $N:=|f|$. For any $F$ with one vertex we have $F = \tdun{a}$ and
	$$\psi (F) = \psi(j_A(a)) = \varphi (a) = \bar\varphi(\tdun{a}).$$
    Thus we must have $N=|f|\geq2$. Then we have two cases:
	\begin{itemize}
		\item[$\bullet$] if $f = B_+^a(\tilde{f}),$ we write $\tilde{f} = t_1\dots t_n,$ then $\psi(t_i)=\bar\varphi(t_i)$ for $i=1,\cdots,n$ and
		\begin{multline*}
		    \psi(f) = \psi(\tdun{a}\diamond B_+^{1_{A}}(\tilde{f})) = \varphi(a)P(\psi(\tilde{f})) = \varphi(a)P(\psi(t_1\dots t_n)) \\
            =\varphi(a)P(\psi(t_1)\dots \psi(t_n)) = \varphi(a)P(\bar{\varphi}(t_1)\dots \bar{\varphi}(t_n)) =  \bar{\varphi}(f);
		\end{multline*}
        which is a contradiction.
		\item[$\bullet$] if $f = t_1\cdots t_n$ then again $\psi(t_i)=\bar\varphi(t_i)$ for $i=1,\cdots,n$ and since $\psi$ and $\bar\varphi$ are both algebra morphisms for the concatenation of rooted trees we find
		$$\psi(f) = \psi(t_1)\cdots \psi(t_n) = \bar{\varphi}(t_1)\cdots \bar{\varphi}(t_n) = \bar{\varphi}(f).$$
	\end{itemize}
	Again, this is a contradiction. Therefore $\psi = \bar{\varphi}$, which concludes the proof. 
\end{proof}
\begin{remark}
 Notice that the property \ref{diagramdiam} cannot be stated directly as a universal property since we need $R$ to ba associative (otherwise $\bar\varphi$ is ill-defined) but $\diamond_\lambda$ is not. Thus we cannot say that $(\calF_A,\diamond_\lambda, B_+^{1_A})$ is freely generated by $A$ as an associative RBA since it does not belong to this category.
\end{remark}

\bibliographystyle{unsrt}
 \addcontentsline{toc}{section}{References}
\bibliography{tree_cone_v5}

\end{document}